\documentclass[11pt]{article}
\usepackage{amssymb,amsmath,mathrsfs,amsthm,indentfirst} 
\usepackage[numbers,sort&compress]{natbib}
\begin{document}

\title{
\bf Stationary solutions of a free boundary problem modeling the growth of vascular tumors with a necrotic core}

\author{Huijuan Song$^a$, Bei Hu$^{b}$ and Zejia Wang$^{a}$
\\
{\small{\it $^a$College of Mathematical and Informational Science, Jiangxi Normal University,
}}
\\
{\small{\it Nanchang 330022, China }}
\\
{\small{\it $^b$ Department of Applied Computational Mathematics and Statistics, University of Notre Dame,}}
\\
{\small{\it Notre Dame, IN 46556, USA}}
}

\date{}

\maketitle

\begin{abstract}
In this paper, we present a rigorous mathematical analysis of a free boundary problem modeling the growth of a vascular solid tumor with a necrotic core. If the vascular system supplies the nutrient concentration $\sigma$ to the tumor at a rate $\beta$, then $\frac{\partial\sigma}{\partial\bf n}+\beta(\sigma-\bar\sigma)=0$ holds on the tumor boundary, where $\bf n$ is the unit outward normal to the boundary and $\bar\sigma$ is the nutrient concentration outside the tumor. The living cells in the nonnecrotic region proliferate at a rate $\mu$.
We show that for any given $\rho>0$, there exists a unique $R\in(\rho,\infty)$ such that the corresponding radially symmetric solution solves the  steady-state necrotic tumor system with necrotic core boundary $r=\rho$ and outer boundary
$r=R$; moreover, there exist a positive integer $n^{**}$ and a sequence of $\mu_n$, symmetry-breaking stationary solutions bifurcate from the radially symmetric stationary solution for
each $\mu_n$ (even $n\ge n^{**})$.
\\
{2000 MR Subject Classification:} 35R35, 35K57, 35B35
\end{abstract}

\let\oldsection\section
\def\SEC{\oldsection}
\renewcommand\section{\setcounter{equation}{0}\SEC}
\renewcommand\thesection{\arabic{section}}

\renewcommand\theequation{\thesection.\arabic{equation}}
\newtheorem{definition}{Definition}[section]
\newtheorem{lemma}{Lemma}[section]
\newtheorem{theorem}{Theorem}[section]
\newtheorem{remark}{Remark}[section]
\def\pd#1#2{\dfrac{\partial#1}{\partial#2}}
\allowdisplaybreaks
\renewcommand{\proofname}{\indent\it\bfseries Proof}
\newcommand{\ep}{\varepsilon}
\newcommand{\eps}[1]{{#1}_{\varepsilon}}

\section{Introduction}

The process of tumor growth in vivo is a complicated phenomenon involving many inter-related processes,
which can be divided into two phases: avascular and vascular growth. In both phases, when a tumor has grown to a detectable size, the inner far-from-surface region of the tumor may comprise only dead cells due to the nonuniform distribution of nutrient materials, which is called necrotic core; see \cite{B-F(05),Cui(06),F-MB(05)} and the references cited therein.
In this paper, we
are interested in a model for the growth of a vascular solid tumor, which consists of a necrotic core of dead cells, and a surrounding shell which contains life-proliferating cells.
Assume that there is no consumption of the nutrient as well as no vasculature in the necrotic core; on the other hand, as a result of angiogenesis, the proliferating rim possesses
its own vasculature, and then the nutrient may be supplied to the nonnecrotic shell via the capillary network. Thus, as in \cite{B-Ch(96),C-F(01)}, the nutrient concentration $\sigma$
satisfies the following reaction-diffusion equation
\begin{equation}
\label{eq(0.1)}
c\frac{\partial \sigma}{\partial t}=\Delta \sigma+[\Gamma(\sigma_B-\sigma)-\lambda_0\sigma]I_{\Omega(t)\setminus D(t)}\quad\mathrm {in} \ \Omega(t),
\end{equation}
where $\Omega(t)\subset{\mathbb R}^3$ is the tumor domain at time $t$ with a moving boundary $\partial\Omega(t)$, $D(t)\subset\Omega(t)$
is the necrotic core region, $\Gamma$ is the transfer rate of nutrient-in-blood-tissue,
$\sigma_B$ is the concentration of nutrient in the vasculature, and so the term $\Gamma(\sigma_B-\sigma)$ accounts for the transfer of nutrient by
means of the vasculature stemming from angiogenesis in the nonnecrotic region. The term $\lambda_0\sigma$ describes the nutrient consumption
by proliferating cells at the rate of $\lambda_0$, and $c=T_{\rm diffusion}/T_{\rm growth}$ represents the ratio of the nutrient diffusion time scale to the tumor
growth (i.e., tumor doubling) time scale. Typically, $T_{\rm diffusion}\thickapprox1$ minute while $T_{\rm growth}\thickapprox1$ day \cite{Byr-Cha,C-F(01)}. Thus, $c\ll1$. Throughout this paper, the notation $I_E$ will always be used to denote the indicator function of a subset $E$ of ${\mathbb R}^3$.

Using appropriate change of variables \cite{C-F(01),Fri-Lam}, we can rewrite
the equation \eqref{eq(0.1)} in the form
\begin{equation}
c\frac{\partial \sigma}{\partial t}=\Delta\sigma-\sigma I_{\Omega(t)\setminus D(t)} \quad \mathrm {in} \ \Omega(t).
\label{eq(0.2)}
\end{equation}
As mentioned above, it is natural to assume that
\begin{equation}\label{eq(0.3)}
\sigma=\underline\sigma\quad \mathrm {in} \ D(t),
\end{equation}
where $\underline\sigma>0$ plays the role of a threshold value in
the sense that in the region where $\sigma>\underline\sigma$ nutrient is enough to sustain (at least a portion of)
 tumor cells alive and proliferating, whereas in the region where $\sigma\le\underline\sigma$, nutrient is insufficient to sustain any tumor cell alive \cite{X-B(15)}. In addition,
since the nutrient enters tumor by the vascular system, $\sigma$ satisfies the boundary condition:
\begin{equation}\label{eq(0.4)}
\frac{\partial\sigma}{\partial\bf n} +\beta(t)(\sigma-\bar\sigma)=0 \quad {\mathrm {on}}\ \partial\Omega(t),
\end{equation}
where $\bf n$ is the outward  normal, $\bar\sigma$ is
the nutrient concentration outside the tumor,  $\beta(t)$ is the rate of nutrient supply to the tumor, which may vary in time.
Angiogenesis results in an increase in $\beta(t)$; conversely, if the tumor is treated with anti-angiogenic drugs, $\beta(t)$ will decrease and the starved
tumor will shrink.

The pressure $p$ stems from the transport of cells which proliferate or die. Let $\bf V$ be the velocity of tumor cells. If we assume that the
density of tumor cells is constant, there is no proliferation in the necrotic core, and the proliferation rate within the nonnecrotic region
 is linearly dependent on the nutrient, then the conservation of mass gives that
$$
{\rm div}{\bf V}= \mu(\sigma-\tilde\sigma)I_{\Omega(t)\setminus D(t)}-\nu I_{D(t)},
$$
where $\mu$ is a positive parameter expressing the ``intensity" of the expansion by mitosis,
the term $\mu(\sigma-\tilde\sigma)$ means that
the cell birthrate is $\mu\sigma$, while the cell death rate (apoptosis) is given by $\mu\tilde\sigma$,
and $\nu$ is the dissolution rate of necrotic cells.
Combining with Darcy's law ${\bf V}=-\nabla p$, we obtain
\begin{equation}\label{eq(0.5)}
-\Delta p=\mu(\sigma-\tilde\sigma) I_{\Omega(t)\setminus D(t)}-\nu I_{D(t)}\quad\mathrm {in} \ \Omega(t).
\end{equation}

Due to cell-to-cell adhesiveness, there holds
\begin{equation}\label{eq(0.6)}
p=\kappa \quad \mathrm {on} \ \partial\Omega(t),
\end{equation}
where $\kappa$ is the mean curvature,
and the continuity of the velocity field up to the boundary of the tumor leads to
\begin{equation}\label{eq(0.7)}
V_{\bf n}={\bf V}\cdot{\bf n}=-\frac{\partial p}{\partial {\bf n}}\quad  \mathrm {on} \ \partial\Omega(t),
\end{equation}
where $V_{\bf n}$ is the velocity of the free boundary in the direction $\bf n$.

The first reaction-diffusion mathematical model of tumor growth in the form of a free boundary problem of a system of partial differential equations was proposed in 1972 by Greenspan \cite{Gr-72,Gr-76}, which was remarkably improved by Byrne and Chaplain \cite{Byr-Cha,B-Ch(96)} during 1990's, where apoptosis is incorporated. It was Friedman and Reitich \cite{Fri-Rei} who made the rigorous mathematical analysis of such free boundary problems in 1999. To date, mathematical modeling, numerical simulation and theoretical analysis on tumor models have been carried out in a large number of papers; see \cite{B-F(05),B-E-Z,Nie,Cui(05),Cui(06),Cui-Es(07),Cui-Es(08),Cui,Cui(19),C-F1,C-F(01),C-Z(18),EM,Fon,F-MB(05),Fri4,F-H1,Fri5,Fri-Hu(07),
Fri-Hu(07-2),Fri6,
Fri-Lam,Hao,HZH(17),HZH(19),SW,SW(18),Wan,WXS,Wei(06),W2,Wu(19),Wu(18),WC2007,WC2015,X-B(15),
Zh-C,Zh-E-C,Zh-W,Z-C(18),Z-C(18-2)}, the review articles \cite{Ara,Fri1,Fri3,Low,Na(05)} and the references therein.

If $D(t)=\emptyset$ and $\Omega(t)=B_{R(t)}(0)$, then the model \eqref{eq(0.2)}-\eqref{eq(0.7)} describes the growth of a nonnecrotic spherical tumor with angiogenesis, proposed by Friedman and Lam \cite{Fri-Lam} as a revision to the Byrne-Chaplain inhibitor-free tumor model \cite{Byr-Cha} in which instead of \eqref{eq(0.4)},
\begin{equation}
\label{eq(0.8)}
 \sigma=\bar\sigma
\end{equation}
is imposed on the tumor boundary. Biologically, the boundary condition \eqref{eq(0.4)} is more reasonable compared with \eqref{eq(0.8)}. In fact, as explained in \cite{Z-C(18)}, tumor surface acts as a barrier to nutrient diffusion, and $1/\beta(t)$ reflects the reduction rate of nutrient by the tumor surface; particularly, $1/\beta(t)=0$ means that tumor surface is obstacle-free to nutrient diffusion which is the case \eqref{eq(0.8)}, whereas $1/\beta(t)=\infty$ means that tumor surface is a complete barrier to nutrient diffusion.

In the nonnecrotic case, assuming $\beta(t)\equiv\beta$ and $0<\tilde\sigma<\bar\sigma$, Friedman and Lam \cite{Fri-Lam} showed that the system \eqref{eq(0.2)}-\eqref{eq(0.7)} allows a unique radially symmetric stationary solution $(\sigma_s(r),p_s(r),R_s)$;
later, Huang, Zhang and Hu \cite{HZH(17)} proved that a branch of symmetry-breaking stationary solutions bifurcates from the radially symmetric stationary solution for each $\mu_n(R_s)$ (even $n\ge 2$) with free boundary
\begin{equation*}
r=R_s+\ep Y_{n,0}(\theta,\varphi) +o(\ep),
\end{equation*}
where $Y_{n,0}$ is the spherical harmonic of order $(n,0)$; very recently, they \cite{HZH(19)} made further efforts and found a threshold value $\mu^*=\mu^*(R_s)$ such that the radially symmetric stationary solution is linearly stable for $\mu<\mu^*$ and linearly unstable for $\mu>\mu^*$ with respect to nonradial perturbations. For the case where the nutrient consumption rate and the proliferation rate of tumor cells are both general functions, Zhuang and Cui established the existence and uniqueness of radially symmetric stationary solutions \cite{Z-C(18)}, and proved that there exists a branch of bifurcation solutions bifurcating from the radially symmetric stationary solution for $\gamma_k$ (any $k\ge k_*$ in the 2 dimension case and even $k\ge k_*$ in the 3 dimension case) by taking $\gamma$ (surface tension coefficient $\gamma$ is defined when \eqref{eq(0.6)} is replaced by $p=\gamma\kappa$) as the bifurcation parameter \cite{C-Z(18)}. The asymptotic stability of radially symmetric stationary solutions was also analyzed in \cite{Z-C(18),Z-C(18-2)}. For cases in the presence of inhibitor, Wang et al. \cite{WXS} obtained the existence of symmetric-breaking stationary solutions for $\mu_n$ (even $n>n^{**}$); also see a very recent paper \cite{SW} for the discussion on the existence of radially symmetric stationary solutions and the asymptotic behavior of radially symmetric transient solutions.
Assuming $\beta(t)=\infty$, i.e., \eqref{eq(0.8)} holds, tumor models have been intensively studied; we refer the reader to \cite{Cui-Es(07),Cui-Es(08),Cui,C-F1,Fon,Fri4,F-H1,Fri5,Fri-Hu(07),Fri-Hu(07-2),Fri6,WC2007,Wan,WC2015,Zh-C,Zh-E-C} and the reference therein.

For the necrotic case, most of studies were on tumor models where the Dirichlet boundary condition \eqref{eq(0.8)} is imposed; see \cite{B-F(05),B-E-Z,B-Ch(96), C-F(01), Cui(05),Cui(06),Cui(19),F-MB(05),Hao,Wei(06),Wu(19),Wu(18)}. Cui \cite{Cui(06)} proved the existence and uniqueness of radially symmetric stationary solutions under a crucial assumption $\underline\sigma<\tilde\sigma<\underline\sigma+\nu/\mu$, i.e., $\nu>\mu(\tilde\sigma-\underline\sigma)>0$,
improving earlier results in \cite{C-F(01)}. Under the assumption $\nu$=0, Hao et al. \cite{Hao} derived the first bifurcation result for the tumor model with a necrotic core, in which they studied the  two-dimensional case by taking $\mu$ as a bifurcation parameter with the aid of numerical calculations. Very recently, Wu \cite{Wu(19)} rigorously analyzed the necrotic multilayered tumor model, and obtained the existence of bifurcation branches of non-flat stationary solutions for $\gamma_k$ ($k\ge K$).
Asymptotic stability of stationary solutions to the above two types of necrotic tumor models was also studied, cf. \cite{Cui(19),Wu(18)}. For necrotic tumor models with \eqref{eq(0.4)},
Shen et al. \cite{SW(18)} established the existence and uniqueness of radially symmetric stationary solutions to the tumor spheroid model with $\nu>0$ under certain conditions on the parameters.

It is well known that the main feature of the free boundary problems modeling the growth of necrotic tumors is that they include two free boundaries, one for the
outer tumor boundary, whose evolution is governed by an evolution equation (such as \eqref{eq(0.7)}), the other for the inner necrotic boundary, an obstacle-type free surface, whose evolution is implicit. Due to the presence of two free
boundaries, the mathematical analysis turns out to be far more challenging.

Motivated by \cite{Hao,Wu(19)}, we shall perform rigorous mathematical analysis of the stationary state of the problem \eqref{eq(0.2)}-\eqref{eq(0.7)} in the
three-dimensional case, under the assumption that
\begin{equation*}
\label{eq(0.9)}
0<\underline\sigma<\tilde\sigma<\bar\sigma=1,\quad\beta(t)\equiv\beta,\quad\nu=0,
\end{equation*}
where $\beta$ is a positive constant. That is,
\begin{align}
&\Delta\sigma=\sigma I_{\Omega\setminus D}\quad{\rm in}~\Omega,\label{eq(1.1)}
\\
-&\Delta p=\mu(\sigma-\tilde{\sigma})I_{\Omega\setminus D}\quad{\rm in}~\Omega,\label{eq(1.2)}
\\
&\sigma=\underline\sigma\quad{\rm in}~D,\label{eq(1.3)}
\\
&[\partial_{\bf n}\sigma]=0\quad{\rm on}~\partial D,\label{eq(1.3.1)}
\\
&[p]=0,\quad[\partial_{\bf n}p]=0\quad{\rm on}~\partial D,\label{eq(1.3.2)}
\\
&\partial_{\bf n}\sigma+\beta(\sigma-1)=0\quad{\rm on}~\partial\Omega,\label{eq(1.4)}
\\
&p=\kappa\quad {\rm on}~\partial\Omega,\label{eq(1.5)}
\\
&\partial_{\bf n}p=0\quad{\rm on}~\partial\Omega.\label{eq(1.6)}
\end{align}
Here and below, the notation $[p]\big|_{\partial D}$ denotes the jump of $p$ as it crosses $\partial D$, i.e.,
$$
[p]=p^+\big|_{\partial D}-p^-\big|_{\partial D}\quad{\rm for}\quad
p^+=p\big|_{\Omega\setminus D}\quad{\rm and}\quad
p^-=p\big|_{D}.
$$
Similarly,
$[\partial_{\bf n}\sigma]\big|_{\partial D}$ and $[\partial_{\bf n}p]\big|_{\partial D}$ denote the jump of the normal derivatives of $\sigma$ and $p$ across $\partial D$ respectively. It is not difficult to find out, the inner boundary conditions \eqref{eq(1.3.1)}, \eqref{eq(1.3.2)} are implied by the equations \eqref{eq(1.1)}, \eqref{eq(1.2)}, and by the maximum principle, $\sigma>\underline\sigma$ in $\Omega\setminus\overline{D}$.
For the problem \eqref{eq(1.1)}-\eqref{eq(1.6)}, we shall first adopt a similar idea as that in \cite{Hao} to study radially symmetric solutions, then
by choosing $\mu$ as a bifurcation parameter, show that there exist a positive integer $n^{**}$ and a sequence of $\mu_n$ such that for
each $\mu_n$ (even $n\ge n^{**})$, a branch of symmetry-breaking stationary solutions bifurcates from the radially symmetric solution (see Theorem \ref{thm-2} stated below).

We stress that in carrying out the bifurcation analysis based on the Crandall-Rabinowitz theorem (see Theorem \ref{thm-3}), there are three major difficulties to be overcome. Firstly, since the problem \eqref{eq(1.1)}-\eqref{eq(1.6)} involves two free boundaries, there is a need for an appropriate Hanzawa-type transformation. Secondly, noting that the expansions \eqref{sigma}, \eqref{p} are needed for the computation of Fr\'echet derivatives, in order to rigorously prove the expansions \eqref{sigma}, \eqref{p} of $\sigma$, $p$ with respect to $\ep$, we have to analyze the dependence of the inner boundary on the outer boundary and the nutrient concentration (see \eqref{T}, \eqref{T-S}), and
establish Schauder estimates in each region (necrotic and nonnecrotic)  for diffraction problems by using the Schauder estimates for an elliptic equation and that near the boundary for elliptic systems \cite{ADN}. Thirdly, it is necessary to verify $B_0<0$ (see Lemma \ref{lem-5.2}). However, as we shall see later, the expression of $B_0$ is complex. Inspired by \cite{HZH(17)}, we will write $B_0$ in order of ascending powers of $\beta$. Then, thanks to the existence of explicit forms for the modified Bessel functions of order of half an odd integer, we can deduce that the coefficient (the function of $\rho$, $R$) of each term is negative by a lengthy calculation.

The rest of this paper is arranged as follows. In Section 2, we present some preliminary material which will be needed in the next sections. In Section 3, we study the radial symmetric solutions to the problem \eqref{eq(1.1)}-\eqref{eq(1.6)}. In Section 4, we analyze the linearization of \eqref{eq(1.1)}-\eqref{eq(1.6)} about the radial solution by introducing an appropriate Hanzawa-type transformation. In the last section, we prove the existence of bifurcation solutions.

\section{Preliminaries}

In this section,  we first collect some properties of the spherical harmonics and the modified spherical Bessel functions, then present an auxiliary lemma and end with the Crandall-Rabinowitz theorem.

In $\mathbb{R}^3$, the family of the spherical harmonic functions $\{Y_{n,m}\}$ forms a complete orthonormal basis for $L^2(\Sigma)$,
where $\Sigma$ is the unit sphere, and
\begin{equation}
\label{Y}
\Delta_\omega Y_{n,m}=-n(n+1)Y_{n,m}.
\end{equation}
Here and below,
$$
\Delta_\omega=\frac1{\sin\theta}\frac{\partial}{\partial\theta}\left(\sin\theta
\frac{\partial}{\partial\theta}\right)
+\frac1{\sin^2\theta}\frac{\partial^2}{\partial\varphi^2}
$$
is the Laplace operator on $\Sigma$. The Laplace operator in $\mathbb{R}^3$ can be written as
$$
\Delta=\frac{\partial^2}{\partial r^2}+\frac2r\frac{\partial}{\partial r}
+\frac1{r^2}\Delta_\omega.
$$

The modified spherical
Bessel functions given by
$$
i_n(s)=\sqrt{\frac{\pi}{2s}}I_{n+1/2}(s),\quad k_n(s)=\sqrt{\frac{\pi}{2s}}K_{n+1/2}(s),\quad s>0,\quad n=0,1,2,\cdots,
$$
form a fundamental solution set of the differential equation
(cf. \citep[(10.47.7)--(10.47.9)]{NIST})
\begin{equation}
\label{B-1}
y''(s)+\frac2s y'(s)-\left[1+\frac{n(n+1)}{s^2}\right]y(s)=0,
\end{equation}
where $I_{n+1/2}(s)$ and $K_{n+1/2}(s)$ are the modified Bessel functions,
$I_\nu(s)>0$, $I'_\nu(s)>0$ for $s>0$ \citep[(10.25.2)]{NIST}, and $K_\nu(s)>0$, $K'_\nu(s)<0$ for $s>0$
\citep[(10.32.9)]{NIST}. Let $g_n(s)$ denote $i_n(s)$ or $(-1)^n k_n(s)$.
Then from \citep[(10.51.4) and (10.51.5)]{NIST},
\begin{align}
&g_{n-1}(s)-g_{n+1}(s)=\frac{2n+1}sg_n(s),\quad n\ge1,\label{B-2}
\\
&g'_n(s)=g_{n-1}(s)-\frac{n+1}sg_n(s),
\quad n\ge1,\label{B-3}
\\
&g'_n(s)=g_{n+1}(s)+\frac{n}sg_n(s),\quad n\ge0.\label{B-4}
\end{align}
It follows from \eqref{B-3}, \eqref{B-4} that
\begin{equation}
\label{B-10}
i'_n(s)>0\quad{\rm and}\quad k'_n(s)<0,\quad n\ge0,~s>0.
\end{equation}
In addition, by \citep[(10.28.2), (10.41.1) and (10.41.2)]{NIST},
\begin{align}
&i_n(s)k_{n+1}(s)+i_{n+1}(s)k_n(s)=\frac{\pi}2\frac1{s^2},\quad n\ge0,
\label{B-9}
\\
&i_n(s)\thicksim\frac1{\sqrt{2(2n+1)s}}\left(\frac{es}{2n+1}
\right)^{n+\frac12}
\quad{\rm as}~n\to\infty,\label{B-7}
\\
&k_n(s)\thicksim\frac{\pi}{\sqrt{2(2n+1)s}}
\left(\frac{es}{2n+1}\right)^{-n-\frac12}
\quad{\rm as}~n\to\infty.\label{B-8}
\end{align}
Explicit expressions for $i_n$, $k_n$, $n=0$, $1$ are as follows \citep[(10.49.9) and (10.49.13)]{NIST}:
\begin{align}
&i_0(s)=\frac{\sinh s}{s},\quad k_0(s)=\frac{\pi}2\frac{e^{-s}}{s},
\label{B-5}
\\
&i_1(s)=-\frac{\sinh s}{s^2}+\frac{\cosh s}{s},
\quad k_1(s)=\frac{\pi}2e^{-s}\left(\frac1s+\frac1{s^2}\right).
\label{B-6}
\end{align}
A direct calculation shows that
\begin{align}
i_0(t)k_0(s)-i_0(s)k_0(t)&=\frac{\pi}2\frac1{ts}\sinh(t-s),\label{i,k-2}
\\
i_1(t)k_1(s)-i_1(s)k_1(t)&=\frac{\pi}2\frac1{t^2s^2}[(t-s)\cosh(t-s)+(ts-1)\sinh(t-s)],
\label{i,k-1}
\\
i_0(t)k_1(s)+i_1(s)k_0(t)&=\frac{\pi}2\frac1{ts^2}[\sinh(t-s)+s
\cosh(t-s)].\label{i,k-3}
\end{align}

In order to establish the dependence of the inner boundary on the tumor outer boundary, we need the following lemma.

\begin{lemma} (\citep[Lemma 8.2]{Fri7})
\label{lem-2.1}
Let $s$ be a nonnegative integer, and let
$$
f(\theta,\varphi)=\sum_{n\ge0,m}f_{n,m}Y_{n,m}(\theta,\varphi).
$$
Then there exist positive constants $c_1$, $c_2$ independent of $f$ such that
$$
c_1\|f\|^2_{H^{s+1/2}(\Sigma)}\le\sum_{n\ge0}(1+n^{2s+1})\sum_m|f_{n,m}|^2\le c_2\|f\|^2_{H^{s+1/2}(\Sigma)}.
$$
\end{lemma}

The bifurcation analysis in the present paper is based on an application of the following theorem of Crandall and Rabinowitz.

\begin{theorem}
\label{thm-3}
(\citep[Theorem 1.7]{Cra}) Let $X$, $Y$ be real Banach spaces and $F(x,\mu)$ a $C^p$ map, $p\ge3$, of a neighborhood $(0,\mu_0)$ in
$X\times\mathbb R$ into $Y$. Suppose

(i) $F(0,\mu)=0$ for all $\mu$ in a neighborhood of $\mu_0$;

(ii) ${\rm Ker}[F_x(0,\mu_0)]$ is a one dimensional space, spanned by $x_0$;

(iii) ${\rm Im}[F_x(0,\mu_0)]=Y_1$ has codimension $1$;

(iv) $[F_{\mu x}](0,\mu_0)x_0\not\in Y_1$.
\\
Then $(0,\mu_0)$ is a bifurcation point of the equation $F(x,\mu)=0$ in the following sense: In a neighborhood of $(0,\mu_0)$ the set of
solutions of $F(x,\mu)=0$ consists of two $C^{p-2}$ smooth curves $\Gamma_1$ and $\Gamma_2$ which intersect only at the point
$(0,\mu_0)$; $\Gamma_1$ is the curve
$(0,\mu)$ and $\Gamma_2$ can be parameterized as follows:
$$
 \Gamma_2: (x(\ep),\mu(\ep)),|\ep|~{\rm small},(x(0),\mu(0))=(0,\mu_0),x'(0) = x_0.
$$
\end{theorem}

\section{Radially symmetric stationary solutions}

In this section, we study radially symmetric solutions to the system \eqref{eq(1.1)}-\eqref{eq(1.6)}, denoted by $(\sigma_s(r),p_s(r),\rho,R)$, where $r=|x|$.

First, it follows from
\eqref{eq(1.1)}, \eqref{eq(1.3)}, \eqref{eq(1.3.1)} and \eqref{eq(1.4)} that
\begin{equation}
\label{eq(2.1)}
\begin{cases}
&\sigma_s''(r)+\frac2r\sigma_s'(r)-\sigma_s(r)=0,\quad\rho<r<R,
\\
&\sigma_s(\rho)=\underline\sigma,\quad\sigma'_s(\rho)=0,\quad\sigma_s'(R)+\beta(\sigma_s(R)-1)=0.
\end{cases}
\end{equation}
Using \eqref{B-1}, \eqref{B-4} and \eqref{B-9}, we find
\begin{equation}
\label{eq(2.2)}
\sigma_s(r)=
\begin{cases}
\frac2{\pi}\underline\sigma\rho^2(k_1(\rho)i_0(r)+i_1(\rho)k_0(r))&\quad{\rm for}~\rho<r<R,
\\
\underline\sigma&\quad{\rm for}~0\le r\le\rho
\end{cases}
\end{equation}
with $\rho$, $R$ satisfying
\begin{equation}
\label{eq(2.3)}
(k_1(\rho)i_1(R)-i_1(\rho)k_1(R))
+\beta(k_1(\rho)i_0(R)+i_1(\rho)k_0(R))
=\frac{\pi\beta}{2\underline\sigma\rho^2}.
\end{equation}

Next we solve for $p_s$. In view of \eqref{eq(1.2)}, \eqref{eq(1.3.2)},
\eqref{eq(1.5)} and \eqref{eq(1.6)}, we find that $p_s$ satisfies
\begin{equation}
\label{eq(2.4)}
-p''_s(r)-\frac2r p'_s(r)=\mu(\sigma_s(r)-\tilde\sigma)\quad{\rm for}~\rho<r<R,
\end{equation}
\begin{equation}
\label{eq(2.5)}
p'_s(\rho)=0,\quad p_s(R)=\frac1R,\quad p'_s(R)=0.
\end{equation}
Let $q=p_s+\mu\sigma_s$. Then \eqref{eq(2.4)} reduces to
$$
\Delta q=\mu\tilde\sigma.
$$
Thus,
\begin{align*}
q(r)&=\frac16\mu\tilde\sigma r^2+C_1+C_2\frac1r,
\\
p_s(r)&=-\mu\sigma_s(r)+\frac16\mu\tilde\sigma r^2+C_1+C_2\frac1r,\quad\rho<r<R,
\end{align*}
where $C_1$, $C_2$ are constants.
Combining with the boundary condition \eqref{eq(2.5)}, we solve for $C_1$, $C_2$ and derive
\begin{equation}
\label{eq(2.21)}
p_s(r)=
\begin{cases}
-\mu(\sigma_s(r)-\sigma_s(R))+\frac16\mu\tilde\sigma(r^2-R^2)
+\frac13\mu\tilde\sigma\rho^3\left(\frac1r-\frac1R\right)+\frac1R\quad{\rm for}~\rho<r<R,
\\
-\mu(\underline\sigma-\sigma_s(R))+\frac16\mu\tilde\sigma(\rho^2-R^2)
+\frac13\mu\tilde\sigma\rho^3\left(\frac1\rho-\frac1R\right)+\frac1R\quad{\rm for}~0\le r\le\rho,
\end{cases}
\end{equation}
where $\rho$, $R$ satisfy
\begin{equation}\label{eq(2.19)}
\tilde\sigma=\frac{3R^2}{R^3-\rho^3}\sigma'_s(R)
=\frac{3R^2}{R^3-\rho^3}\frac{2\underline\sigma\rho^2}{\pi}
(k_1(\rho)i_1(R)-i_1(\rho)k_1(R)).
\end{equation}

In the sequel, we should solve the system of equations \eqref{eq(2.3)} and \eqref{eq(2.19)} with respect to the variables $\rho$ and $R$, for given $\beta$, $\underline\sigma$ and $\tilde\sigma$. However, for technical reasons, we will instead compute $R$, $\tilde\sigma$ for given $\rho$, $\beta$, $\underline\sigma$, as was done in \cite{Hao}. More precisely, we will establish the following result.

\begin{theorem}
\label{thm-1}
Assume $\mu>0$, $\beta>0$ and $0<\underline\sigma<1$. If we fix $\rho>0$, then there exist a unique $R$ in $(\rho,\infty)$
and a unique $\tilde\sigma$ in $(\underline\sigma,1)$ such that $(\sigma_s(r),p_s(r),\rho,R)$ is a radially symmetric solution to the problem \eqref{eq(1.1)}-\eqref{eq(1.6)}, where $\sigma_s(r)$, $p_s(r)$ are given by \eqref{eq(2.2)}, \eqref{eq(2.21)}.
\end{theorem}

\begin{proof}
In view of \eqref{eq(2.3)}, we consider the function
$$
f(s)=(k_1(\rho)i_1(s)-i_1(\rho)k_1(s))
+\beta(k_1(\rho)i_0(s)+i_1(\rho)k_0(s))
-\frac{\pi\beta}{2\underline\sigma\rho^2},\quad s\ge\rho.
$$
Then from \eqref{B-4}, \eqref{B-9} it follows that
\begin{align*}
&f(\rho)=\frac{\pi\beta}{2\rho^2}\left(1-\frac1{\underline\sigma}\right),
\\
&f'(s)=(k_1(\rho)i'_1(s)-i_1(\rho)k'_1(s))+\beta(i_1(s)k_1(\rho)-i_1(\rho)k_1(s)).
\end{align*}
Since $\beta>0$ and $0<\underline\sigma<1$, $f(\rho)<0$. By \eqref{B-10} we know that $f'(s)>0$ for $s>\rho$.
Thus, combining with $\lim_{s\to+\infty}f(s)=+\infty$, we deduce that there exists a unique $R>\rho$ such that \eqref{eq(2.3)} holds.
In what follows, we prove that $\tilde\sigma$ given by \eqref{eq(2.19)} satisfies
\begin{equation}\label{eq(2.7)}
\underline\sigma<\tilde\sigma<\sigma_s(R)<1.
\end{equation}
First, we show that $\underline\sigma<\tilde\sigma$, which reduces to
\begin{equation}
\label{eq(2.20)}
3(R-\rho)\cosh(R-\rho)+3(R\rho-1)\sinh(R-\rho)-R^3+\rho^3>0
\end{equation}
by \eqref{i,k-1}. As a matter of fact, define
$$
g(s)=3(s-\rho)\cosh(s-\rho)+3(s\rho-1)\sinh(s-\rho)-s^3+\rho^3,\quad s\ge\rho.
$$
Then, clearly, $g(\rho)=0$ and
\begin{equation*}
g'(s)=3s\big\{\rho[\cosh(s-\rho)-1]+[\sinh(s-\rho)-(s-\rho)]\big\}.
\end{equation*}
Since, as may easily be verified,
\begin{equation}
\label{sq-1}
\cosh z>\frac{\sinh z}z>1,\quad z>0,
\end{equation}
$g'(s)>0$ for $s>\rho$. Thus, $g(s)>0$ for every $s>\rho$, i.e., \eqref{eq(2.20)} is valid.

We next proceed to verify that $\tilde\sigma<\sigma_s(R)$, which is equivalent to
\begin{equation}
\label{eq(2.17)}
(R^3-\rho^3)(k_1(\rho)i_0(R)+i_1(\rho)k_0(R))>
3R^2(k_1(\rho)i_1(R)-i_1(\rho)k_1(R))
\end{equation}
by \eqref{eq(2.2)} and \eqref{eq(2.19)}.
Based on \eqref{i,k-1} and \eqref{i,k-3}, \eqref{eq(2.17)} can be further
simplified to the form
\begin{equation}
\label{eq(2.18)}
(\rho+1)(R^3-3R^2+3R-\rho^3)e^{2(R-\rho)}
+(\rho-1)(R^3+3R^2+3R-\rho^3)>0.
\end{equation}
Consider the function
$$
h(s)=(\rho+1)(s^3-3s^2+3s-\rho^3)e^{2(s-\rho)}
+(\rho-1)(s^3+3s^2+3s-\rho^3)\quad{\rm for}~s\ge\rho.
$$
A direct calculation shows that
\begin{align*}
h'(s)&=(\rho+1)(2s^3-3s^2-2\rho^3+3)e^{2(s-\rho)}
+3(\rho-1)(s+1)^2,
\\
h''(s)&=2(\rho+1)(2s^3-3s-2\rho^3+3)e^{2(s-\rho)}+6(\rho-1)(s+1),
\\
h'''(s)&=2(\rho+1)(4s^3+6s^2-6s-4\rho^3+3)e^{2(s-\rho)}+6(\rho-1),
\\
h^{(4)}(s)&=16(\rho+1)(s^3+3s^2-\rho^3)e^{2(s-\rho)}>0\quad{\rm for}~s>\rho,
\end{align*}
and in particular,
$$
h(\rho)=h'(\rho)=h''(\rho)=0,\quad h'''(\rho)=12\rho^3.
$$
Applying Taylor's formula with Lagrange's remainder-term
$$
h(s)=h(\rho)+h'(\rho)(s-\rho)+\frac{h''(\rho)}{2!}(s-\rho)^2+
\frac{h'''(\rho)}{3!}(s-\rho)^3+\frac{h^{(4)}(\xi)}{4!}(s-\rho)^4,
$$
where $\xi\in(\rho,s)$, provided that $s>\rho$, we see that $h(s)>0$
for $s>\rho$. Hence, \eqref{eq(2.18)} follows.

Finally, we prove that $\sigma_s(R)<1$. Indeed, from \eqref{eq(2.1)} and the fact $\sigma_s'(R)>0$, it is easy to see that
$$
\sigma_s(R)=1-\frac{\sigma_s'(R)}\beta<1.
$$
The proof is complete.
\end{proof}


\section{Linearized problem}

In this section, we study the linearization of the problem \eqref{eq(1.1)}-\eqref{eq(1.6)} at the radially symmetric solution $(\sigma_s,p_s,\rho,R)$, where
$(\rho,R)$ satisfies \eqref{eq(2.3)}, \eqref{eq(2.19)}.
In order to compute Fr\'echet derivatives in the next section, we shall derive rigorous mathematical estimates
for this linearization.

Consider a family of domains with perturbed boundaries
\begin{align*}
\partial\eps{D}: r&=\rho+\ep T(\theta,\varphi),
\\
\partial\eps{\Omega}: r&=R+\ep S(\theta,\varphi),
\end{align*}
where both $S$ and $T$ are unknown functions.
Let $(\sigma,p)$ be the solution to
\begin{align}
&\Delta\sigma=\sigma I_{\eps{\Omega}\setminus\eps{D}}\quad {\rm in}~\eps{\Omega},\label{eq(3.3)}
\\
-&\Delta p=\mu(\sigma-\tilde{\sigma})I_{\eps{\Omega}\setminus\eps{D}}\quad {\rm in}~\eps{\Omega},\label{eq(3.4)}
\\
&\sigma=\underline\sigma,\quad[\partial_{\bf n}\sigma]=0 \quad{\rm on}~\partial\eps{D},\label{eq(3.5)}
\\
&\partial_{\bf n}\sigma+\beta(\sigma-1)=0\quad {\rm on}~\partial\eps{\Omega},\label{eq(3.6)}
\\
&[p]=0,\quad[\partial_{\bf n}p]=0 \quad{\rm on}~\partial\eps{D},\label{eq(3.7)}
\\
&p=\kappa\quad {\rm on}~\partial\eps{\Omega}.\label{eq(3.8)}
\end{align}
Then the well-posedness of the problem \eqref{eq(3.3)}-\eqref{eq(3.8)} follows from the local flattening and
diffraction problem for systems; see \cite{LU}.
If we define
\begin{equation}\label{F}
F(\tilde{R},\mu)=\frac{\partial p}{\partial\bf n}\Big|_{\partial\eps{\Omega}},
\end{equation}
where $\tilde{R}=\ep S$,
then $(\sigma,p,\rho+\ep T,R+\ep S)$ is a solution to the system \eqref{eq(1.1)}-\eqref{eq(1.6)} if and only
if $F(\tilde R,\mu)=0$.

Let us formally write
\begin{equation}\label{sigma}
\sigma(r,\theta,\varphi)=
\begin{cases}
\hat{\sigma}(r)+\ep\sigma_1(r,\theta,\varphi)+O(\ep^2)&\quad{\rm in}~\eps{\Omega}\setminus\eps{D},
\\
\hat{\sigma}(\rho)+\ep\sigma_1(r,\theta,\varphi)+O(\ep^2)&\quad{\rm in}~\eps{D},
\end{cases}
\end{equation}
\begin{equation}\label{p}
p(r,\theta,\varphi)=
\begin{cases}
\hat{p}(r)+\ep p_1(r,\theta,\varphi)+O(\ep^2)&\quad{\rm in}~\eps{\Omega}\setminus\eps{D},
\\
\hat{p}(\rho)+\ep p_1(r,\theta,\varphi)+O(\ep^2)&\quad{\rm in}~\eps{D},
\end{cases}
\end{equation}
where
$$
\hat{\sigma}(r)=\frac2{\pi}\underline\sigma\rho^2(k_1(\rho)i_0(r)+i_1(\rho)k_0(r)),
~0<r<\infty,
$$
$$
\hat{p}(r)=-\mu(\hat\sigma(r)-\hat\sigma(R))+\frac16\mu\tilde\sigma(r^2-R^2)
+\frac13\mu\tilde\sigma\rho^3\left(\frac1r-\frac1R\right)+\frac1R,~0<r<\infty.
$$
From \eqref{eq(2.2)} and \eqref{eq(2.21)} we can easily see that
\begin{equation}
\label{hatsigma}
\Delta\hat{\sigma}=\hat{\sigma}\quad{\rm in}~{\mathbb R}^3\setminus\{0\},\quad
\hat{\sigma}(r)=\sigma_s(r)\quad{\rm for}~\rho<r<R,
\end{equation}
\begin{equation}
\label{hatp}
-\Delta\hat{p}=\mu(\hat{\sigma}-\tilde\sigma)\quad{\rm in}~{\mathbb R}^3\setminus\{0\},\quad
\hat{p}(r)=p_s(r)\quad{\rm for}~\rho<r<R.
\end{equation}

In the following we shall first present a formal derivation of the linearized problem for $(\sigma_1,p_1)$.
Recall that the gradient operator in $\mathbb{R}^3$ can be written as
$$
\nabla=\vec{e}_r\partial_r+\vec{e}_\theta\frac1r\partial_\theta+
\vec{e}_\varphi\frac1{r\sin\theta}\partial_\varphi=\vec{e}_r\partial_r+
\frac1r\nabla_\omega,
$$
and if a surface is given by $r=R+\ep S(\theta,\varphi)$, then
\begin{equation*}
{\bf n}=\vec{e}_r-\frac{\ep}R\nabla_\omega S+\ep^2f_1,
\end{equation*}
where $\|f_1\|_{C^{l-1+\alpha}(\Sigma)}\le C_l$ provided that $\|S\|_{C^{l+\alpha}(\Sigma)}\le1$.
Thus, from \eqref{eq(3.6)}, \eqref{sigma}, \eqref{hatsigma} and \eqref{eq(2.1)}, we derive
\begin{equation*}
\label{sigma-1}
\begin{split}
0=&\left(\frac{\partial\sigma}{\partial r}\bigg|_{\partial\eps{\Omega}}\vec{e}_r+\frac1R\nabla_\omega\sigma\bigg|_{\partial\eps{\Omega}}\right)\cdot
\left(\vec{e}_r-\frac{\ep}R\nabla_\omega S\right)+\beta\sigma\bigg|_{\partial\eps{\Omega}}-\beta+O(\ep^2)
\\
=&\frac{\partial\sigma}{\partial r}\bigg|_{\partial\eps{\Omega}}
-\frac{\ep}{R^2}\nabla_\omega\sigma\bigg|_{\partial\eps{\Omega}}\cdot\nabla_\omega S
+\beta\sigma\bigg|_{\partial\eps{\Omega}}-\beta+O(\ep^2)
\\
=&\frac{\partial\hat{\sigma}}{\partial r}\bigg|_{\partial\eps{\Omega}}+\ep\frac{\partial\sigma_1}{\partial r}\bigg|_{\partial\eps{\Omega}}+\beta\hat{\sigma}\bigg|_{\partial\eps{\Omega}}
+\beta\ep\sigma_1\bigg|_{\partial\eps{\Omega}}-\beta+O(\ep^2)
\\
=&\frac{\partial\sigma_s}{\partial r}\bigg|_{\partial B_R}+\frac{\partial^2\sigma_s}{\partial r^2}\bigg|_{\partial B_R}\ep S+\ep\frac{\partial\sigma_1}{\partial r}\bigg|_{\partial B_R}+\beta\sigma_s\bigg|_{\partial B_R}+\beta\frac{\partial\sigma_s}{\partial r}\bigg|_{\partial B_R}\ep S+\beta\ep\sigma_1\bigg|_{\partial B_R}
\\
&-\beta+O(\ep^2)
\\
=&\ep\left[\left(\frac{\partial\sigma_1}{\partial r}+\beta\sigma_1\right)\bigg|_{\partial B_R}+\lambda S\right]+O(\ep^2),
\end{split}
\end{equation*}
where
$$
\lambda=\bigg(\frac{\partial^2\sigma_s}{\partial r^2}+\beta\frac{\partial\sigma_s}{\partial r}\bigg)\bigg|_{\partial B_R}.
$$
Using \eqref{eq(2.1)}, \eqref{eq(2.2)} and \eqref{B-4}, we get
$$
\frac{\partial\sigma_s}{\partial r}\bigg|_{\partial B_R}=\frac2{\pi}\underline\sigma\rho^2
(k_1(\rho)i_1(R)-i_1(\rho)k_1(R)),
$$
and
$$
\frac{\partial^2\sigma_s}{\partial r^2}\bigg|_{\partial B_R}=\frac2{\pi}\underline\sigma\rho^2
\left[(i_0(R)k_1(\rho)+i_1(\rho)k_0(R))-\frac2R(i_1(R)k_1(\rho)-i_1(\rho)k_1(R))\right],
$$
so that
\begin{equation}\label{lambda}
\lambda=\frac2{\pi}\underline\sigma\rho^2
\bigg[i_0(R)k_1(\rho)+i_1(\rho)k_0(R)
+\bigg(\beta-\frac2R\bigg)(i_1(R)k_1(\rho)
-i_1(\rho)k_1(R))\bigg].
\end{equation}
Substituting \eqref{sigma} into the boundary condition \eqref{eq(3.5)} and using \eqref{hatsigma}, \eqref{eq(2.1)} again, we have
\begin{align*}
0=&\sigma^+\big|_{\partial\eps{D}}-\underline\sigma
=\hat{\sigma}\big|_{\partial\eps{D}}+\ep\sigma_1^+\big|_{\partial\eps{D}}-\underline\sigma+O(\ep^2)
\\
=&\sigma_s(\rho)+\sigma_s'(\rho)\ep T+\ep\sigma_1^+\big|_{\partial B_\rho}-\underline\sigma+O(\ep^2)
\\
=&\ep\sigma_1^+\big|_{\partial B_\rho}+O(\ep^2),
\\
0=&\sigma^-\big|_{\partial\eps{D}}-\underline\sigma
=\hat{\sigma}(\rho)+\ep\sigma_1^-\big|_{\partial\eps{D}}-\underline\sigma+O(\ep^2)
\\
=&\ep\sigma_1^-\big|_{\partial B_\rho}+O(\ep^2).
\end{align*}
Analogously,
\begin{align*}
0=&\frac{\partial\sigma^+}{\partial{\bf n}}\bigg|_{\partial\eps{D}}
=\frac{\partial\sigma^+}{\partial r}\bigg|_{\partial\eps{D}}
-\frac{\ep}{\rho^2}\nabla_\omega\sigma^+\bigg|_{\partial\eps{D}}\cdot\nabla_\omega T
+O(\ep^2)
\\
=&\frac{\partial\hat{\sigma}}{\partial r}\bigg|_{\partial\eps{D}}
+\ep\frac{\partial\sigma_1^+}{\partial r}\bigg|_{\partial\eps{D}}
+O(\ep^2)
\\
=&\hat{\sigma}'(\rho)+\hat{\sigma}''(\rho)\ep T+\ep\frac{\partial\sigma_1^+}{\partial r}\bigg|_{\partial B_\rho}
+O(\ep^2)
\\
=&\ep\left(\frac{\partial\sigma_1^+}{\partial r}\bigg|_{\partial B_\rho}+\underline\sigma T\right)
+O(\ep^2).
\end{align*}
Hence,  $\sigma_1$ satisfies:
\begin{align}
&-\Delta\sigma_1=0\quad{\rm in}~B_\rho,\label{sigma-1-1}
\\
&-\Delta\sigma_1+\sigma_1=0\quad{\rm in}~B_R\setminus\overline{B}_\rho,\label{sigma-1-2}
\\
&\sigma_1=0,\quad\frac{\partial\sigma_1^+}{\partial r}=-\underline\sigma T\quad{\rm on}~\partial B_\rho,\label{sigma-1-3}
\\
&\frac{\partial\sigma_1}{\partial r}+\beta\sigma_1=-\lambda S
\quad{\rm on}~\partial B_R.\label{sigma-1-4}
\end{align}

We now expand the boundary conditions for $p$.
By the expression from \cite{Fri7}
\begin{equation*}
\kappa\bigg|_{\partial\eps{\Omega}}=
\frac1R-\frac{\ep}{R^2}\left(S+\frac12\Delta_\omega S\right)
+\ep^2 f_2,
\end{equation*}
where $\|f_2\|_{C^{l-2+\alpha}(\Sigma)}\le C_l$ if $\|S\|_{C^{l+\alpha}(\Sigma)}\le1$, we deduce from \eqref{eq(3.8)}, \eqref{p}, \eqref{hatp} and \eqref{eq(2.5)} that
\begin{equation*}
\begin{aligned}
0=&\hat{p}\bigg|_{\partial\eps{\Omega}}+\ep p_1\bigg|_{\partial\eps{\Omega}}-\frac1R+\frac{\ep}{R^2}\left(S+\frac12\Delta_\omega S\right)
+O(\ep^2)
\\
=&p_s\bigg|_{\partial B_R}+\frac{\partial p_s}{\partial r}\bigg|_{\partial B_R}\ep S+\ep p_1\bigg|_{\partial B_R}
-\frac1R+\frac{\ep}{R^2}\left(S+\frac12\Delta_\omega S\right)
+O(\ep^2)
\\
=&\ep\left[p_1\bigg|_{\partial B_R}+\frac1{R^2}\left(S+\frac12\Delta_\omega S\right)\right]+O(\ep^2).
\end{aligned}
\end{equation*}
The  boundary condition \eqref{eq(3.7)} together with \eqref{p}, \eqref{hatp} and \eqref{eq(2.5)} leads to
\begin{equation*}
0=p^+\big|_{\partial\eps{D}}-p^-\big|_{\partial\eps{D}}=\hat{p}\big|_{\partial\eps{D}}
+\ep p_1^+\big|_{\partial\eps{D}}-\hat{p}(\rho)-\ep p_1^-\big|_{\partial\eps{D}}+O(\ep^2)
=\ep[p_1]\big|_{\partial B_\rho}+O(\ep^2),
\end{equation*}
and
\begin{align*}
0=&\frac{\partial p^+}{\partial{\bf n}}\bigg|_{\partial\eps{D}}
-\frac{\partial p^-}{\partial{\bf n}}\bigg|_{\partial\eps{D}}
\\
=&\frac{\partial p^+}{\partial r}\bigg|_{\partial\eps{D}}
-\frac{\partial p^-}{\partial r}\bigg|_{\partial\eps{D}}-\frac{\ep}{\rho^2}\left(\nabla_\omega p^+\bigg|_{\partial\eps{D}}\cdot\nabla_\omega T-\nabla_\omega p^-\bigg|_{\partial\eps{D}}\cdot\nabla_\omega T\right)
+O(\ep^2)
\\
=&\frac{\partial\hat{p}}{\partial r}\bigg|_{\partial\eps{D}}
+\ep\frac{\partial p_1^+}{\partial r}\bigg|_{\partial\eps{D}}
-\ep\frac{\partial p_1^-}{\partial r}\bigg|_{\partial\eps{D}}
+O(\ep^2)
\\
=&\frac{\partial^2\hat{p}}{\partial r^2}\bigg|_{\partial B_\rho}\ep T+\ep\frac{\partial p_1^+}{\partial r}\bigg|_{\partial B_\rho}
-\ep\frac{\partial p_1^-}{\partial r}\bigg|_{\partial B_\rho}
+O(\ep^2)
\\
=&\ep\left([\partial_{\bf n}p_1]\big|_{\partial B_\rho}+\mu(\tilde{\sigma}
-\underline\sigma)T\right)+O(\ep^2).
\end{align*}
Therefore,  $p_1$ satisfies
\begin{align}
&-\Delta p_1=0\quad{\rm in}~B_\rho,\label{p-1-1}
\\
&-\Delta p_1=\mu\sigma_1\quad{\rm in}~B_R\setminus\overline{B}_\rho,\label{p-1-2}
\\
&[p_1]=0,\quad
[\partial_{\bf n}p_1]
=-\mu(\tilde\sigma-\underline\sigma)T\quad{\rm on}~\partial B_\rho,\label{p-1-3}
\\
&p_1=-\frac1{R^2}\bigg(S+\frac12\Delta_\omega S\bigg)\quad{\rm on}~\partial B_R.\label{p-1-4}
\end{align}

We next compute explicitly the functions $\sigma_1$, $p_1$ and $T$. If we write
\begin{equation}
\label{S}
S(\theta,\varphi)=\sum_{n=0}^\infty\sum_{m=-n}^n a_{n,m} Y_{n,m}(\theta,\varphi),
\end{equation}
then solving \eqref{sigma-1-2}-\eqref{sigma-1-4} by the separation of variables and using \eqref{Y}, we get
\begin{equation}
\label{sigma-1+}
\sigma_1(r,\theta,\varphi)=-\lambda\sum_{n=0}^\infty\sum_{m=-n}^n a_{n,m} Q_n(r)Y_{n,m}(\theta,\varphi)\quad{\rm in}~B_R\setminus\overline{B}_\rho,
\end{equation}
where $Q_n(r)$ satisfies
\begin{align}
&Q''_n(r)+\frac2rQ'_n(r)-\left(1+\frac{n(n+1)}{r^2}\right)Q_n(r)=0,\quad \rho<r<R,\label{Qn-equ}
\\
&Q_n(\rho)=0,\quad Q'_n(R)+\beta Q_n(R)=1,\label{Qn-bou}
\end{align}
i.e.,
\begin{equation}
\label{Qn}
Q_n(r)=\frac{i_n(r)k_n(\rho)-i_n(\rho)k_n(r)}
{\left(\frac{n}R+\beta\right)(i_n(R)k_n(\rho)-i_n(\rho)k_n(R))+i_n(\rho)k_{n+1}(R)
+i_{n+1}(R)k_n(\rho)}
\end{equation}
for $\rho<r<R$. Substituting \eqref{sigma-1+} into the second equation in \eqref{sigma-1-3} yields
\begin{equation}
\label{T}
T(\theta,\varphi)
=\frac{\lambda}{\underline\sigma}\sum_{n=0}^\infty\sum_{m=-n}^n a_{n,m} Q'_n(\rho)Y_{n,m}(\theta,\varphi).
\end{equation}
Similarly, setting
\begin{equation}
\label{p1}
p_1(r,\theta,\varphi)=\sum_{n=0}^\infty\sum_{m=-n}^n a_{n,m}P_n(r)Y_{n,m}(\theta,\varphi)
\end{equation}
in \eqref{p-1-1}-\eqref{p-1-4}, we are led to
\begin{align*}
&
P''_n(r)+\frac2rP'_n(r)-\frac{n(n+1)}{r^2}P_n(r)=0\quad{\rm for}~0<r<\rho,
\\
&
P''_n(r)+\frac2rP'_n(r)-\frac{n(n+1)}{r^2}P_n(r)=\lambda\mu Q_n(r)\quad{\rm for}~\rho<r<R,
\\
&P'_0(0)=0,\quad P_n(0)=0\quad{\rm for}~n\ge1,
\\
&P_n(\rho+)=P_n(\rho-),\quad
P'_n(\rho+)-P'_n(\rho-)=-\lambda\mu\frac{\tilde\sigma-\underline\sigma
}{\underline\sigma}Q_n'(\rho),
\\
&P_n(R)=-\frac{1}{R^2}\left(1-\frac{n(n+1)}2\right).
\end{align*}
Solving the ODE, we obtain
\begin{equation}
\label{Pn-}
P_n(r)=
\frac{\lambda\mu\tilde\sigma\rho^{n+2}Q_n'(\rho)}{(2n+1)\underline\sigma}
\left(\frac1{\rho^{2n+1}}-\frac1{R^{2n+1}}\right)r^n
-\lambda\mu\frac{r^n}{R^n}Q_n(R)
-\frac{r^n}{R^{n+2}}\left(1-\frac{n(n+1)}2\right)
\end{equation}
for $0<r\le\rho$, and
\begin{equation}
\label{Pn+}
\begin{split}
P_n(r)=&
\frac{\lambda\mu\tilde\sigma\rho^{n+2}Q_n'(\rho)}{(2n+1)\underline\sigma}
\left(\frac1{r^{2n+1}}-\frac1{R^{2n+1}}\right)r^n
+\lambda\mu\left(Q_n(r)-\frac{r^n}{R^n}Q_n(R)\right)
\\
&-\frac{r^n}{R^{n+2}}\left(1-\frac{n(n+1)}2\right)\quad{\rm for}~\rho<r<R.
\end{split}
\end{equation}

To proceed further, we need more information on $Q_n$.

\begin{lemma}
\label{lem-1}
$\{Q_n'(\rho)\}$ is a positive, monotonically decreasing sequence in $n$, while $\{Q_n'(R)\}$ is a positive, monotonically increasing sequence in $n$.
\end{lemma}

\begin{proof}
From \eqref{Qn} it follows that
\begin{equation*}
Q'_n(r)=\frac{i'_n(r)k_n(\rho)-i_n(\rho)k'_n(r)}
{\left(\frac{n}R+\beta\right)(i_n(R)k_n(\rho)-i_n(\rho)k_n(R))+i_n(\rho)k_{n+1}(R)
+i_{n+1}(R)k_n(\rho)},
\end{equation*}
which together with \eqref{B-10} implies that $Q'_n(r)>0$, $\rho\le r\le R$; particularly,
$Q_n'(\rho)>0$, $Q_n'(R)>0$. Moreover, from \eqref{Qn-bou},
\begin{equation}
\label{Qn-1}
0=Q_n(\rho)<Q_n(r)\le Q_n(R)<\frac1{\frac{n}R+\beta}\quad{\rm for}~\rho<r\le R.
\end{equation}
Let $F_n(r)=Q_{n+1}(r)-Q_n(r)$. Then by virtue of \eqref{Qn-equ} and \eqref{Qn-bou}, $F_n$ satisfies
\begin{equation*}
\begin{cases}
-\Delta F_n+\left(1+\frac{n(n+1)}{r^2}\right)F_n=-\frac{2(n+1)}{r^2}Q_{n+1}\quad{\rm in}~B_R\setminus\overline{B}_\rho,\\
F_n=0\quad{\rm on}~\partial B_\rho,\\
\frac{\partial F_n}{\partial{\bf n}}+\beta F_n=0\quad{\rm on}~\partial B_R.
\end{cases}
\end{equation*}
By \eqref{Qn-1}, we apply the maximum principle to conclude that
$F_n(r)<0$ for $\rho<r\le R$, $F'_n(\rho)<0$ and $F'_n(R)>0$. The proof is complete.
\end{proof}

\begin{remark}
\label{rem-4.1}
In fact, from the proof of Lemma \ref{lem-1}, one can also see that
for every $\rho<r\le R$, $\{Q_n(r)\}$ is a positive, monotonically decreasing sequence in $n$,
and $\lim_{n\to\infty}Q_n(r)=0$ by \eqref{Qn-1}. In particular,
$$
\lim_{n\to\infty}Q_n(R)=0.
$$
As a result,
\begin{equation}
\label{Qn-2}
\lim_{n\to\infty}Q_n'(R)=\lim_{n\to\infty}(1-\beta Q_n(R))=1.
\end{equation}
\end{remark}

\begin{lemma}
\label{lem-2}
Given $\rho<r\le R$,
$\{Q_n'(r)/Q_n(r)-n/r\}$
is a positive, monotonically decreasing sequence in $n$. Moreover,
\begin{equation}
\label{Qn-3}
\lim_{n\to\infty}\left(\frac{Q_n'(R)}{Q_n(R)}-\frac n{R}\right)=0.
\end{equation}
\end{lemma}

\begin{proof}
For given $s\in(\rho,R]$, from \eqref{Qn-equ}, \eqref{Qn-bou} we see that $G_n(r)$ defined by
$$
G_n(r)=\frac rsQ_n(r)-\frac{Q_n(s)}{Q_{n+1}(s)}Q_{n+1}(r),
$$
satisfies
\begin{equation*}
\begin{cases}
&-\Delta G_n+\left[1+\frac{(n+1)(n+2)}{r^2}\right]G_n=-\frac{2}{s}
\left(Q_n'(r)-\frac{n}{r}Q_n(r)\right)\quad{\rm in}~B_s\setminus\overline{B}_\rho,\\
&G_n=0\quad{\rm on}~\partial B_\rho,\quad G_n=0\quad{\rm on}~\partial B_s.
\end{cases}
\end{equation*}
A direct calculation based on \eqref{Qn}, \eqref{B-4} gives that
\begin{equation}
\label{Qn-4}
\begin{split}
&Q_n'(r)
\\
=&\frac{\frac{n}r[i_n(r)k_n(\rho)-i_n(\rho)k_n(r)]
+[i_{n+1}(r)k_n(\rho)+i_n(\rho)k_{n+1}(r)]}
{\left(\frac{n}R+\beta\right)[i_n(R)k_n(\rho)-i_n(\rho)k_n(R)]+i_n(\rho)k_{n+1}(R)
+i_{n+1}(R)k_n(\rho)}
\\
=&\frac{n}r Q_n(r)+\frac{i_{n+1}(r)k_n(\rho)+i_n(\rho)k_{n+1}(r)}
{\left(\frac{n}R+\beta\right)[i_n(R)k_n(\rho)-i_n(\rho)k_n(R)]+i_n(\rho)k_{n+1}(R)
+i_{n+1}(R)k_n(\rho)},
\end{split}
\end{equation}
which implies that
\begin{equation*}
Q_n'(r)-\frac{n}r Q_n(r)>0,\quad \rho\le r\le R.
\end{equation*}
Thus, $G_n(r)<0$ for $\rho<r<s$ and $G'_n(s)>0$ by the maximum principle.
Combining with the definition of $G_n(r)$, we further have
$$
0<\frac{Q_{n+1}'(s)}{Q_{n+1}(s)}-\frac{n+1}{s}<\frac{Q_n'(s)}{Q_n(s)}-\frac n{s},
$$
which proves the monotonicity result.

Now we turn to the proof of \eqref{Qn-3}. Using \eqref{Qn}, \eqref{Qn-4}, we compute
\begin{equation*}
\frac{Q_n'(R)}{Q_n(R)}-\frac n{R}
=\frac{i_{n+1}(R)k_n(\rho)+i_n(\rho)k_{n+1}(R)}{i_n(R)k_n(\rho)-i_n(\rho)k_n(R)}
=\frac{\frac{i_{n+1}(R)}{i_n(R)}+\frac{i_n(\rho)}{i_n(R)}\frac{k_{n+1}(R)}{k_n(\rho)}}
{1-\frac{i_n(\rho)}{i_n(R)}\frac{k_n(R)}{k_n(\rho)}}.
\end{equation*}
By \eqref{B-7}, \eqref{B-8} and the fact $0<\rho<R$, letting $n\to\infty$ in the above equality yields \eqref{Qn-3}. The proof is complete.
\end{proof}

We conclude this section by rigorously establishing the expansions \eqref{sigma}, \eqref{p} by estimating the $O(\ep^2)$ terms in the $C^{2+\alpha}$-norm. To do that, observing that $(\sigma,p)$ is defined in $\eps{\Omega}$ while $(\sigma_1,p_1)$ is defined only in $B_R$, we will first transform the domain $B_R$ into $\eps{\Omega}$ by introducing the following Hanzawa-type transformation
\begin{equation}
\label{Hanzawa}
(r,\theta,\varphi)=\eps{H}(r',\theta',\varphi')=(r'+\chi(\rho-r')\ep T(\theta',\varphi')+\chi(R-r')\ep S(\theta',\varphi'),\theta',\varphi'),
\end{equation}
where
$$
\chi\in C^{\infty},  \quad \chi(z)=\left\{
\begin{array}{c}
0,\quad{\rm if}~|z|\ge3\delta_0/4
\\
1,\quad{\rm if}~|z|<\delta_0/4
\end{array}
\right.,
\quad \left|\frac{d^k\chi}{dz^k}\right|\le\frac{C}{\delta_0^k}
$$
with $\delta_0$ being a small positive constant satisfying $\delta_0<\min\{\frac43\rho,
\frac23(R-\rho)\}$.
It is easy to see that $\eps{H}$ maps $B_R$ ($B_\rho$) into $\eps{\Omega}$ ($\eps{D}$) while keeping
the ball $\{r<\rho-\frac34\delta_0\}$ fixed, and the inverse transformation $\eps{H}^{-1}$ maps $\eps{\Omega}$ ($\eps{D}$) into $B_R$ ($B_\rho$). Let
$$
\tilde{\sigma}_1(r,\theta,\varphi)=\sigma_1(\eps{H}^{-1}(r,\theta,\varphi))\quad{\rm in}~\eps{\Omega},
$$
\begin{equation}
\label{tilde-p-1}
\tilde{p}_1(r,\theta,\varphi)=p_1(\eps{H}^{-1}(r,\theta,\varphi))\quad{\rm in}~\eps{\Omega}.
\end{equation}
Then $(\tilde{\sigma}_1,\tilde{p}_1)$ is well defined in $\eps{\Omega}$.

Next, if $S$, given by \eqref{S}, belongs to $C^{4+\alpha}(\Sigma)$, and is $\pi$-periodic in $\theta$ and $2\pi$-periodic in $\varphi$, then $T$ admits the representation \eqref{T} and satisfies
$$
\|T\|_{C(\Sigma)}\le C\|T\|_{H^{3/2}(\Sigma)}\le C\|S\|_{H^{3/2}(\Sigma)}
$$
by Lemmas \ref{lem-2.1} and \ref{lem-1}. Furthermore,
\begin{equation}
\label{T-S}
\|T\|_{C^{2+\alpha}(\Sigma)}\le C\|T\|_{C^3(\Sigma)}\le C\|S\|_{H^{9/2}(\Sigma)}\le C\|S\|_{C^{4+\alpha}(\Sigma)},
\end{equation}
provided that $1/2<\alpha<1$.

We now proceed to establish the following.

\begin{lemma}
\label{lem-sch}
Assume that $S\in C^{4+\alpha}(\Sigma)$ with $1/2<\alpha<1$ and $\|S\|_{C^{4+\alpha}(\Sigma)}\le1$, and is given by \eqref{S}, $\pi$-periodic in $\theta$ and $2\pi$-periodic in $\varphi$. Then the estimates
\begin{align}
\|\sigma-\hat{\sigma}(\rho)-\ep\tilde{\sigma}_1\|_{C^{2+\alpha}(\eps{\overline{D}})}
+\|\sigma-\hat{\sigma}-\ep\tilde{\sigma}_1\|_{C^{2+\alpha}
(\overline{\eps{\Omega}\setminus\eps{D}})}&\le C|\ep|^2\|S\|_{C^{4+\alpha}(\Sigma)},
\label{sigma-sch}
\\
\|p-\hat{p}(\rho)-\ep\tilde{p}_1\|_{C^{2+\alpha}(\eps{\overline{D}})}
+\|p-\hat{p}-\ep\tilde{p}_1\|_{C^{2+\alpha}
(\overline{\eps{\Omega}\setminus\eps{D}})}&\le C|\ep|^2\|S\|_{C^{4+\alpha}(\Sigma)}
\label{p-sch}
\end{align}
are valid uniformly for small $|\ep|$
with $C$ being independent of $\ep$ and $S$.
\end{lemma}

\begin{proof}
By \eqref{sigma-1-1}-\eqref{sigma-1-4}, we see that $\sigma_1\equiv0$ in $B_\rho$.  Applying the Schauder estimates, we obtain
\begin{equation*}\label{sigma-1-sch}
\|\sigma_1\|_{C^{2+\alpha}(\overline{B_R\setminus B_\rho})}\le C\|S\|_{C^{1+\alpha}(\Sigma)}.
\end{equation*}
Moreover, we derive for $\tilde{\sigma}_1$ the problem
\begin{align*}
&\tilde{\sigma}_1\equiv0\quad{\rm in}~\eps{D},
\\
&-\Delta\tilde{\sigma}_1+\tilde{\sigma}_1=\tilde{f}_1\quad{\rm in}~
\eps{\Omega}\setminus\eps{D},
\\
&\tilde{\sigma}_1=0\quad{\rm on}~\partial\eps{D},
\\
&\frac{\partial\tilde{\sigma}_1}{\partial{\bf n}}+\beta\tilde{\sigma}_1=-\lambda S+\tilde{g}_1\quad{\rm on}~\partial\eps{\Omega},
\end{align*}
where $\tilde{f}_1/\ep$ involves at most second order derivatives of $T$, $S$ and $\sigma_1$, $\tilde{g}_1/\ep$ involves at most first order derivatives of $S$ and $\sigma_1$, and
\begin{align}
\|\tilde{f}_1\|_{C^\alpha(\overline{\eps{\Omega}\setminus\eps{D}})}&\le C|\ep|(\|T\|_{C^{2+\alpha}(\Sigma)}+\|S\|_{C^{2+\alpha}(\Sigma)}),\label{f}
\\
\|\tilde{g}_1\|_{C^{1+\alpha}(\Sigma)}&\le C|\ep|\|S\|_{C^{2+\alpha}(\Sigma)}.\label{g}
\end{align}

Set
\begin{equation*}
\psi=
\begin{cases}
\sigma-\hat{\sigma}-\ep\tilde{\sigma}_1
&\quad{\rm in}~\eps{\Omega}\setminus\eps{D},
\\
\sigma-\hat{\sigma}(\rho)-\ep\tilde{\sigma}_1
&\quad{\rm in}~\eps{D}.
\end{cases}
\end{equation*}
Then, we find
\begin{align*}
&\psi\equiv0\quad{\rm in}~\eps{D},
\\
&-\Delta\psi+\psi=-\ep\tilde{f}_1\quad{\rm in}~\eps{\Omega}\setminus\eps{D},
\\
&\psi^+=\hat{\sigma}(\rho)-\hat{\sigma}(\rho+\ep T)\quad
{\rm on}~\partial\eps{D},
\\
&\frac{\partial\psi}{\partial\bf n}+\beta\psi=-\ep\tilde{g}_1+\tilde{g}_2
\quad{\rm on}~\partial\eps{\Omega},
\end{align*}
where
\begin{equation}
\label{inner-bound}
\|\hat{\sigma}(\rho)-\hat{\sigma}(\rho+\ep T)\|_{C^{2+\alpha}(\Sigma)}\le C|\ep|^2\|T\|_{C^{2+\alpha}(\Sigma)},
\end{equation}
and similarly to $\tilde{k}_1$ in \citep[Lemma 3.1]{HZH(17)},
\begin{equation}
\label{g-2}
\|\tilde{g}_2\|_{C^{1+\alpha}(\Sigma)}\le C|\ep|^2\|S\|_{C^{2+\alpha}(\Sigma)}.
\end{equation}
Combining \eqref{f}-\eqref{g-2}, we get
\begin{equation}\label{psi}
\|\psi\|_{C^{2+\alpha}(\overline{\eps{\Omega}\setminus\eps{D}})}\le C|\ep|^2(\|T\|_{C^{2+\alpha}(\Sigma)}+\|S\|_{C^{2+\alpha}(\Sigma)}),
\end{equation}
which together with \eqref{T-S} proves \eqref{sigma-sch}.

We now proceed with the proof of \eqref{p-sch}.
It should be pointed out that, being different from that for $\sigma$ (which equals to a constant in the dead-core region), here we need establish Schauder estimates for diffraction problems by using the estimates near the boundary for elliptic systems \cite{ADN}.
Precisely, for the problem \eqref{p-1-1}-\eqref{p-1-4} there holds
\begin{equation}\label{p-1-sch}
\|p_1\|_{C^{2+\alpha}(\overline{B}_\rho)}+\|p_1\|_{C^{2+\alpha}(\overline{B_R\setminus B_\rho})}\le C(\|T\|_{C^{2+\alpha}(\Sigma)}+\|S\|_{C^{4+\alpha}(\Sigma)}).
\end{equation}
Denote
\begin{equation*}
\label{phi}
\phi=
\begin{cases}
p-\hat{p}-\ep\tilde{p}_1&\quad{\rm in}~\eps{\Omega}\setminus\eps{D},
\\
p-\hat{p}(\rho)-\ep\tilde{p}_1&\quad{\rm in}~\eps{D}.
\end{cases}
\end{equation*}
By a similar procedure as before, we obtain
\begin{align*}
&-\Delta\phi=-\ep\tilde{f}_2\quad{\rm in}~\eps{D},
\\
&-\Delta\phi=\mu(\sigma-\hat{\sigma}-\ep\tilde{\sigma}_1)-\ep\tilde{f}_3
\quad{\rm in}~\eps{\Omega}\setminus\eps{D},
\\
&[\phi]=\hat{p}(\rho)-\hat{p}(\rho+\ep T),\quad
[\partial_{\bf n}\phi]=\tilde{g}_3\quad{\rm on}~\partial\eps{D},
\\
&\phi=\tilde{g}_4\quad{\rm on}~\partial\eps{\Omega},
\end{align*}
where
\begin{align*}
\|\tilde{f}_2\|_{C^\alpha(\eps{\overline{D}})}
+\|\tilde{f}_3\|_{C^\alpha(\overline{\eps{\Omega}\setminus\eps{D}})}&\le C|\ep|(\|T\|_{C^{2+\alpha}(\Sigma)}+\|S\|_{C^{4+\alpha}(\Sigma)}),
\\
\|\hat{p}(\rho)-\hat{p}(\rho+\ep T)\|_{C^{2+\alpha}(\Sigma)}&\le C|\ep|^2\|T\|_{C^{2+\alpha}(\Sigma)},
\\
\|\tilde{g}_3\|_{C^{1+\alpha}(\Sigma)}&\le C|\ep|^2\|T\|_{C^{2+\alpha}(\Sigma)},
\\
\|\tilde{g}_4\|_{C^{2+\alpha}(\Sigma)}&\le C|\ep|^2\|S\|_{C^{4+\alpha}(\Sigma)}.
\end{align*}
Using \eqref{psi}, the Schauder estimates for diffraction problems yield
\begin{equation*}
\|\phi\|_{C^{2+\alpha}(\eps{\overline{D}})}+\|\phi\|_{C^{2+\alpha}(\overline{\eps{\Omega}\setminus\eps{D}})}\le C|\ep|^2(\|T\|_{C^{2+\alpha}(\Sigma)}+\|S\|_{C^{4+\alpha}(\Sigma)}),
\end{equation*}
and then \eqref{p-sch} follows immediately from \eqref{T-S}. The proof is complete.
\end{proof}

\section{Symmetry-breaking solutions}

In this section, regarding \eqref{eq(1.1)}-\eqref{eq(1.6)} as a bifurcation problem with a bifurcation parameter $\mu$ (see \eqref{F}), we prove the existence of symmetry-breaking solutions by using the Crandall-Rabinowitz theorem.

By \eqref{T-S} and \eqref{p-1-sch},
$$
\frac{\partial(\hat{p}+\ep\tilde{p}_1)}{\partial{\bf n}}\bigg|_{\partial\eps{\Omega}}
=\ep\left[\frac{\partial^2\hat{p}}{\partial r^2}\bigg|_{\partial B_R}S
+\frac{\partial p_1}{\partial r}\bigg|_{\partial B_R}\right]+O(|\ep|^2\|S\|_{C^{4+\alpha}(\Sigma)}),
$$
which together with \eqref{hatp} and \eqref{p-sch} leads to
\begin{equation}\label{F-sch}
F(\tilde{R},\mu)=\ep\left[\frac{\partial^2 p_s}{\partial r^2}\bigg|_{\partial B_R}S
+\frac{\partial p_1}{\partial r}\bigg|_{\partial B_R}\right]+O(|\ep|^2\|S\|_{C^{4+\alpha}(\Sigma)}).
\end{equation}

As in \cite{Fon,HZH(17)}, we introduce the Banach spaces
$$
X^{l+\alpha}=\{\tilde{R}\in C^{l+\alpha}(\Sigma),\tilde{R}~{\rm is}~\pi{\rm-periodic~in}
~\theta,2\pi{\rm -periodic~in}~\varphi\},
$$
$$
X_2^{l+\alpha}={\rm closure~of~the~linear~space~spanned~by}~\{Y_{n,0}(\theta),n=0,2,4,\dots\}~{\rm in}~X^{l+\alpha}.
$$
Take $X=X_2^{4+\alpha}$ and $Y=X_2^{1+\alpha}$ with $1/2<\alpha<1$. Noticing that $Y_{n,0}(\pi-\theta)=Y_{n,0}(\theta)$ if and only if
$n$ is even, $X_2^{l+\alpha}$ coincides with the subspace of the $C^{l+\alpha}(\Sigma)$-closure of the smooth functions consisting of those functions
$u$ that are independent of $\varphi$ and
satisfy $u(\theta) =u(\pi-\theta)$. Thus, $F$ maps $X$ into $Y$. The relation \eqref{F-sch} shows that the mapping $(\tilde{R},\mu)\to F(\tilde{R},\mu)$ from $X_2^{l+3+\alpha}$
to $X_2^{l+\alpha}$ is bounded if $l=1$, and the same argument shows that the same is
true for any $l\ge1$. A similar argument shows that this mapping is Fr\'echet differentiable in $(\tilde{R},\mu)$; furthermore $\partial F(\tilde{R},\mu)/\partial\tilde{R}$ (or $\partial F(\tilde{R},\mu)/\partial\mu$) is obtained
by solving a linearized problem about $(\tilde{R},\mu)$ with respect to $\tilde{R}$ (or $\mu$). By
using the Schauder estimates we can then further obtain differentiability of
$F(\tilde{R},\mu)$ to any order.

In view of \eqref{F-sch}, the Fr\'echet derivative of $F(\tilde R,\mu)$
in $\tilde R$ at $(0,\mu)$ is given as follows:
\begin{equation}
\label{eq(4.2)}
[F_{\tilde{R}}(0,\mu)]S=\frac{\partial^2 p_s}{\partial r^2}\bigg|_{\partial B_R}S
+\frac{\partial p_1}{\partial r}\bigg|_{\partial B_R}.
\end{equation}
Using \eqref{eq(2.4)}, \eqref{eq(2.5)}, we have
\begin{equation}
\label{eq(4.3)}
\frac{\partial^2 p_s}{\partial r^2}\bigg|_{\partial B_R}=-\mu(\sigma_s(R)-\tilde\sigma),
\end{equation}
and by \eqref{p1}, \eqref{Pn+},
\begin{equation}
\label{eq(4.4)}
\begin{split}
\frac{\partial p_1}{\partial r}\bigg|_{\partial B_R}
=&\sum_{n=0}^\infty\sum_{m=-n}^n a_{n,m}
\bigg[\frac n{R^3}\left(\frac{n(n+1)}2-1\right)+\lambda\mu\bigg(Q'_n(R)-\frac{n}RQ_n(R)\bigg)
\\
&\quad\quad\quad\quad\quad\quad\quad
-\lambda\mu\frac{\tilde\sigma}{\underline\sigma}\frac{\rho^{n+2}}{R^{n+2}}Q_n'(\rho)\bigg]Y_{n,m}(\theta,\varphi).
\end{split}
\end{equation}
Substituting \eqref{eq(4.3)} and \eqref{eq(4.4)} into \eqref{eq(4.2)}, we arrive at
\begin{equation*}
\begin{aligned}
&[F_{\tilde{R}}(0,\mu)]S=\sum_{n=0}^\infty\sum_{m=-n}^n a_{n,m}
\bigg[\frac n{R^3}\left(\frac{n(n+1)}2-1\right)
+\lambda\mu\bigg(Q'_n(R)-\frac{n}RQ_n(R)
\bigg)
\\
&\quad\quad\quad\quad\quad\quad\quad\quad\quad\quad\quad\quad\quad
-\lambda\mu\frac{\tilde\sigma}{\underline\sigma}\frac{\rho^{n+2}}{R^{n+2}}Q_n'(\rho)
-\mu(\sigma_s(R)-\tilde\sigma)\bigg]Y_{n,m}(\theta,\varphi).
\end{aligned}
\end{equation*}
In particular,
\begin{equation*}
[F_{\tilde{R}}(0,\mu)]Y_{n,m}=(A_n-\mu B_n)Y_{n,m},
\end{equation*}
where
\begin{equation}
\label{eq(4.22)}
A_n=\frac n{R^3}\left(\frac{n(n+1)}2-1\right),
\end{equation}
\begin{equation}
\label{eq(4.6)}
B_n=\sigma_s(R)-\tilde\sigma
-\lambda\left[\left(Q'_n(R)-\frac{n}RQ_n(R)\right)-
\frac{\tilde\sigma}{\underline\sigma}\frac{\rho^{n+2}}{R^{n+2}}Q_n'(\rho)\right] .
\end{equation}
Obviously,
$$
A_0=A_1=0,\quad A_n>0\quad{\rm for}~n\ge2.
$$
As will be shown in the next lemma, $B_0<0$ and $B_1=0$.

\begin{lemma}
\label{lem-5.2}
$B_0<0$ and $B_1=0$.
\end{lemma}

\begin{proof}
Combining \eqref{eq(2.2)}, \eqref{eq(2.19)}, \eqref{lambda} and \eqref{eq(4.6)}, we derive
\begin{equation}
\label{eq(4.10)}
B_n=\frac2{\pi}\underline\sigma\rho^2H_n
\end{equation}
with
\begin{equation}
\label{eq(4.11)}
\begin{aligned}
H_n=&
i_0(R)k_1(\rho)+i_1(\rho)k_0(R)
-\frac{3R^2}{R^3-\rho^3}(i_1(R)k_1(\rho)-i_1(\rho)k_1(R))
\\
&-\left[i_0(R)k_1(\rho)+i_1(\rho)k_0(R)+\left(\beta-\frac2R\right)
(i_1(R)k_1(\rho)-i_1(\rho)k_1(R))\right]
\\
&\quad
\left[\left(Q'_n(R)-\frac{n}R Q_n(R)\right)-\frac{\tilde\sigma}{\underline\sigma}
\frac{\rho^{n+2}}{R^{n+2}}\left(Q_n'(\rho)-\frac{n}\rho Q_n(\rho)\right)\right],
\end{aligned}
\end{equation}
where we have employed the fact that $Q_n(\rho)=0$. Besides, \eqref{eq(2.19)} also implies
\begin{equation}
\label{eq(4.14)}
\frac{\tilde\sigma}{\underline\sigma}=\frac{3R^2}{R^3-\rho^3}
\frac2{\pi}\rho^2(i_1(R)k_1(\rho)-i_1(\rho)k_1(R)),
\end{equation}
and \eqref{Qn-4} together with \eqref{B-9} gives
\begin{equation}\label{Q-R}
\begin{split}
&Q'_n(R)-\frac{n}R Q_n(R)
\\
=&
\frac{i_{n+1}(R)k_n(\rho)+i_n(\rho)k_{n+1}(R)}
{\left(\frac{n}R+\beta\right)[i_n(R)k_n(\rho)-i_n(\rho)k_n(R)]+i_n(\rho)k_{n+1}(R)
+i_{n+1}(R)k_n(\rho)},
\end{split}
\end{equation}
\begin{equation}\label{Q-rho}
\begin{split}
&Q'_n(\rho)-\frac{n}\rho Q_n(\rho)
\\
=&
\frac{i_{n+1}(\rho)k_n(\rho)+i_n(\rho)k_{n+1}(\rho)}
{\left(\frac{n}R+\beta\right)[i_n(R)k_n(\rho)-i_n(\rho)k_n(R)]+i_n(\rho)k_{n+1}(R)
+i_{n+1}(R)k_n(\rho)}
\\
=&
\frac{\pi}{2\rho^2}\frac1
{\left(\frac{n}R+\beta\right)[i_n(R)k_n(\rho)-i_n(\rho)k_n(R)]+i_n(\rho)k_{n+1}(R)
+i_{n+1}(R)k_n(\rho)}.
\end{split}
\end{equation}
Thus, substituting \eqref{eq(4.14)}, \eqref{Q-R} and \eqref{Q-rho} into \eqref{eq(4.11)}, we get
\begin{equation}
\label{eq(4.18)}
H_n=\frac1{\left(\frac{n}R+\beta\right)[i_n(R)k_n(\rho)-i_n(\rho)k_n(R)]+i_n(\rho)k_{n+1}(R)
+i_{n+1}(R)k_n(\rho)}W_n,
\end{equation}
where
\begin{equation}
\label{eq(4.12)}
\begin{aligned}
W_n
=&\left[\left(\frac{n}R+\beta\right)(i_n(R)k_n(\rho)-i_n(\rho)k_n(R))+(i_n(\rho)k_{n+1}(R)
+i_{n+1}(R)k_n(\rho))\right]
\\
&\left[(i_0(R)k_1(\rho)+i_1(\rho)k_0(R))
-\frac{3R^2}{R^3-\rho^3}(i_1(R)k_1(\rho)-i_1(\rho)k_1(R))\right]
\\
&-\left[(i_0(R)k_1(\rho)+i_1(\rho)k_0(R))+\left(\beta-\frac2R\right)
(i_1(R)k_1(\rho)-i_1(\rho)k_1(R))\right]
\\
&\left[(i_n(\rho)k_{n+1}(R)+i_{n+1}(R)k_n(\rho))
-\frac{3R^2}{R^3-\rho^3}\frac{\rho^{n+2}}{R^{n+2}}(i_1(R)k_1(\rho)-i_1(\rho)k_1(R))\right]
\\
=&M_{1,n}+M_{2,n}\beta
\end{aligned}
\end{equation}
with
\begin{equation}
\label{eq(4.16)}
\begin{split}
M_{1,n}=&\frac nR(i_n(R)k_n(\rho)-i_n(\rho)k_n(R))
\\
&\quad\left[(i_0(R)k_1(\rho)+i_1(\rho)k_0(R))
-\frac{3R^2}{R^3-\rho^3}(i_1(R)k_1(\rho)-i_1(\rho)k_1(R))\right]
\\
&+\frac{i_1(R)k_1(\rho)-i_1(\rho)k_1(R)}{R(R^3-\rho^3)}
\bigg[3R^3\frac{\rho^{n+2}}{R^{n+2}}(i_0(R)k_1(\rho)+i_1(\rho)k_0(R))
\\
&\quad\quad\quad\quad\quad\quad\quad\quad\quad\quad\quad\quad
-6R^2\frac{\rho^{n+2}}{R^{n+2}}(i_1(R)k_1(\rho)-i_1(\rho)k_1(R))
\\
&\quad\quad\quad\quad\quad\quad\quad\quad\quad\quad\quad\quad
-(R^3+2\rho^3)(i_n(\rho)k_{n+1}(R)+i_{n+1}(R)k_n(\rho))\bigg],
\end{split}
\end{equation}
\begin{equation}
\label{eq(4.17)}
\begin{aligned}
M_{2,n}=&(i_n(R)k_n(\rho)-i_n(\rho)k_n(R))
\\
&\left[(i_0(R)k_1(\rho)+i_1(\rho)k_0(R))
-\frac{3R^2}{R^3-\rho^3}(i_1(R)k_1(\rho)-i_1(\rho)k_1(R))\right]
\\
&-(i_1(R)k_1(\rho)-i_1(\rho)k_1(R))
\\
&\left[(i_n(\rho)k_{n+1}(R)+i_{n+1}(R)k_n(\rho))
-\frac{3R^2}{R^3-\rho^3}\frac{\rho^{n+2}}{R^{n+2}}(i_1(R)k_1(\rho)-i_1(\rho)k_1(R))\right].
\end{aligned}
\end{equation}

In the sequel, we shall first prove that $B_1=0$. By \eqref{eq(4.10)}, \eqref{eq(4.18)} and \eqref{eq(4.12)},
it suffices to verify that $M_{1,1}=0$ and $M_{2,1}=0$.
As a matter of fact, by virtue of \eqref{eq(4.16)}, \eqref{eq(4.17)}, we compute
\begin{align*}
M_{1,1}=&\frac{R^3+2\rho^3}{R(R^3-\rho^3)}(i_1(R)k_1(\rho)-i_1(\rho)k_1(R))
\\
&~~~~~~~~~~~
\bigg[k_1(\rho)\left(i_0(R)-\frac3Ri_1(R)-i_2(R)\right)
+i_1(\rho)\left(k_0(R)+\frac3Rk_1(R)-k_2(R)\right)\bigg],
\\
M_{2,1}=&(i_1(R)k_1(\rho)-i_1(\rho)k_1(R))
\\
&~~~~~~~~~~~
\bigg[k_1(\rho)\left(i_0(R)-\frac3Ri_1(R)-i_2(R)\right)
+i_1(\rho)\left(k_0(R)+\frac3Rk_1(R)-k_2(R)\right)\bigg].
\end{align*}
From \eqref{B-2},
$$
i_2(s)=i_0(s)-\frac3si_1(s),\quad
k_2(s)=k_0(s)+\frac3sk_1(s),
$$
from which, $M_{1,1}=M_{2,1}=0$ follows.

Next, we show $B_0<0$ by showing that $M_{1,0}<0$ and $M_{2,0}<0$.
In view of \eqref{eq(4.16)},
\begin{equation*}
\label{}
\begin{aligned}
M_{1,0}=&\frac{i_1(R)k_1(\rho)-i_1(\rho)k_1(R)}{R(R^3-\rho^3)}
\bigg[3R\rho^2(i_0(R)k_1(\rho)+i_1(\rho)k_0(R))
\\
&\quad
-6\rho^2(i_1(R)k_1(\rho)-i_1(\rho)k_1(R))
-(R^3+2\rho^3)(i_0(\rho)k_1(R)+i_1(R)k_0(\rho))\bigg].
\end{aligned}
\end{equation*}
Denote
\begin{equation*}
\begin{aligned}
\xi(s)=&3Rs^2(i_0(R)k_1(s)+i_1(s)k_0(R))
-6s^2(i_1(R)k_1(s)-i_1(s)k_1(R))
\\
&-(R^3+2s^3)(i_0(s)k_1(R)+i_1(R)k_0(s))\quad{\rm for}~0<s\leq R.
\end{aligned}
\end{equation*}
Then, by \eqref{B-3}, \eqref{B-4},
\begin{equation*}
\xi'(s)=(R^3+2s^3)(i_1(R)k_1(s)-i_1(s)k_1(R))-
3Rs^2(i_0(R)k_0(s)-i_0(s)k_0(R))
\end{equation*}
for $0<s<R$. Furthermore, observing that
$$
R^3+2s^3=(R-s)^2(R+2s)+3Rs^2,
$$
and applying \eqref{i,k-2} and \eqref{i,k-1}, we get
\begin{equation*}
\begin{aligned}
\xi'(s)>&3Rs^2[(i_1(R)k_1(s)-i_1(s)k_1(R))-(i_0(R)k_0(s)-i_0(s)k_0(R))]
\\
=&\frac{3\pi}{2R}[(R-s)\cosh(R-s)-\sinh(R-s)]
\end{aligned}
\end{equation*}
for $0<s<R$. Thus, by \eqref{sq-1}, we deduce that $\xi'(s)>0$ for $0<s<R$, which together with $\xi(R)=0$ implies
that $\xi(s)<0$, $0<s<R$; in particular, $\xi(\rho)<0$. Hence, $M_{1,0}<0$.

It remains to show that $M_{2,0}<0$. By \eqref{eq(4.17)},
\begin{equation*}
\begin{aligned}
M_{2,0}=&(i_0(R)k_0(\rho)-i_0(\rho)k_0(R))
\\
&~~~~~~~~~~~~~\left[(i_0(R)k_1(\rho)+i_1(\rho)k_0(R))
-\frac{3R^2}{R^3-\rho^3}(i_1(R)k_1(\rho)-i_1(\rho)k_1(R))\right]
\\
&-(i_1(R)k_1(\rho)-i_1(\rho)k_1(R))
\\
&~~~~~~~~~~~~~\left[(i_0(\rho)k_1(R)+i_1(R)k_0(\rho))
-\frac{3\rho^2}{R^3-\rho^3}(i_1(R)k_1(\rho)-i_1(\rho)k_1(R))\right].
\end{aligned}
\end{equation*}
Using \eqref{i,k-2}-\eqref{i,k-3}, we further obtain
\begin{equation*}
\begin{aligned}
M_{2,0}=&\frac{\pi^2}{8R^4\rho^3(R^3-\rho^3)}
\big[-2R^5+3R^4\rho-2R^3-R^2\rho^3+3R^2\rho+2\rho^3-3\rho
\\
&-(R^4-2R^3\rho-6R^2\rho^2+8R\rho^3+6R\rho-\rho^4-6\rho^2)\sinh(2(R-\rho))
\\
&-(R^4\rho-2R^3-3R^2\rho^3-3R^2\rho+2R\rho^4+12R\rho^2-4\rho^3-3\rho)\cosh(2(R-\rho))\big].
\end{aligned}
\end{equation*}
Set
\begin{equation*}
\begin{aligned}
\eta(s)=&[s^4-8Rs^3+6(R^2+1)s^2+2(R^3-3R)s-R^4]\sinh(2(R-s))
\\
&+[-2Rs^4+(3R^2+4)s^3-12Rs^2+(-R^4+3R^2+3)s+2R^3]\cosh(2(R-s))
\\
&-(R^2-2)s^3+(3R^4+3R^2-3)s-2R^5-2R^3\quad{\rm for}~0<s\le R.
\end{aligned}
\end{equation*}
Then
\begin{equation}
\label{eq(4.20)}
M_{2,0}=\frac{\pi^2}{8R^4\rho^3(R^3-\rho^3)}\eta(\rho).
\end{equation}
Notice that
\begin{equation}
\label{eq(4.26)}
\eta''(s)=2s\zeta(s)
\end{equation}
with
\begin{equation*}
\begin{aligned}
\zeta(s)=&[2s^3+(-6R^2-6)s+4R^3+12R]\sinh(2(R-s))
\\
&+[-4Rs^3+6R^2s^2+12Rs-2R^4-9R^2-6]\cosh(2(R-s))
-3(R^2-2),
\end{aligned}
\end{equation*}
and
\begin{equation}
\label{eq(4.27)}
\zeta'(s)=e^{-2(R-s)}w(s),
\end{equation}
where
\begin{align*}
w(s)=&[(4R-2)s^3+(-6R^2-6R+3)s^2+(12R^2-12R+6)s
+2R^4-4R^3+6R^2
\\
&-6R+3]e^{4(R-s)}-(4R+2)s^3+(6R^2-6R-3)s^2+(12R^2+12R+6)s
\\
&-2R^4-4R^3-6R^2-6R-3.
\end{align*}
It is easy to see that
\begin{equation}
\label{eq(4.28)}
\eta(R)=\eta'(R)=\zeta(R)=0.
\end{equation}
Thus, we only need to verify that $w(s)>0$ for $0<s<R$, which together with \eqref{eq(4.26)}-\eqref{eq(4.28)} implies
$$
\eta(s)<0\quad{\rm for}~0<s<R.
$$
In particular, $\eta(\rho)<0$ and by \eqref{eq(4.20)}, there holds $M_{2,0}<0$.

We consider, instead of $w(s)$,
\begin{align*}
u(R)=&[2R^4-4R^3+(-6s^2+12s+6)R^2+(4s^3-6s^2-12s-6)R-2s^3+3s^2
\\
&+6s+3]e^{4(R-s)}
-2R^4-4R^3+(6s^2+12s-6)R^2
\\
&-(4s^3+6s^2-12s+6)R-2s^3-3s^2+6s-3
\end{align*}
for $R\ge s>0$. A direct calculation gives that
$$
u(s)=u'(s)=u''(s)=0,\quad u^{(3)}(s)=240s^2>0,\quad u^{(4)}(s)=16(120s^2+48s)>0,
$$
and
$$
u^{(5)}(R)=64e^{4(R-s)}v(R)
$$
with
\begin{align*}
v(R)=&32R^4+96R^3+(-96s^2+192s+96)R^2+(64s^3-336s^2+288s+24)R
\\
&+48s^3-192s^2+96s+3.
\end{align*}
Furthermore, there hold
$$
v(s)=192s^2+120s+3>0,\quad v'(s)=8(42s^2+60s+3)>0
$$
and
$$
v''(R)=192(2R^2+3R-s^2+2s+1)>0\quad{\rm for}~R>s>0.
$$
On the basis of the above analysis, we conclude that $u(R)>0$ for $R>s>0$, i.e., $w(s)>0$ for $0<s<R$.
The proof is complete.
\end{proof}

\begin{lemma}
\label{lem-5.1}
(i) $\lim_{n\to\infty}B_n=\sigma_s(R)-\tilde\sigma$;

(ii) There exists $n^*\in\mathbb{N}$ such that $\mu_n=A_n/B_n$ is positive and monotonically increasing for $n\ge n^*$; moreover, $\lim_{n\to\infty}\mu_n=+\infty$.
\end{lemma}

\begin{proof}
(i) Applying Remark \ref{rem-4.1} and Lemma \ref{lem-2}, we get
\begin{equation}
\label{eq(4.23)}
\lim_{n\to\infty}\left(Q'_n(R)-\frac{n}R Q_n(R)\right)=
\lim_{n\to\infty}Q_n(R)
\left(\frac{Q'_n(R)}{Q_n(R)}-\frac{n}R\right)=0.
\end{equation}
Besides, from Lemma \ref{lem-1} and $0<\rho<R$, we see that
$$
\lim_{n\to\infty}Q_n'(\rho)
\frac{\rho^{n+2}}{R^{n+2}}=0.
$$
The assertion (i) then follows from the definition of $B_n$.

(ii) Since $\sigma_s(R)-\tilde\sigma>0$ by \eqref{eq(2.7)}, the assertion (i) implies that there exists a positive integer $N_1\ge2$ such that $B_n>0$ for every $n\ge N_1$. Consequently,
\begin{equation}
\label{eq(4.21)}
\mu_n=\frac{A_n}{B_n}
\end{equation}
is well defined and positive for $n\ge N_1$, and
$$
\lim_{n\to\infty}\mu_n=+\infty,\quad
\lim_{n\to\infty}\frac{\mu_n}{n^3}=\frac{1}{2R^3(\sigma_s(R)-\tilde\sigma)}.
$$
We claim that
\begin{equation}
\label{eq(4.5)}
\lim_{n\to\infty}\frac{\frac{\mu_n}{n^3}-\frac{1}{2R^3(\sigma_s(R)-\tilde\sigma)}}{\frac1n}=
\frac{1}{2R^3(\sigma_s(R)-\tilde\sigma)}.
\end{equation}
In fact, using \eqref{eq(4.22)}, \eqref{eq(4.6)} and \eqref{eq(4.21)}, we compute for $n\ge N_1$,
\begin{equation}
\label{eq(4.24)}
\begin{aligned}
\frac{\frac{\mu_n}{n^3}-\frac{1}{2R^3(\sigma_s(R)-\tilde\sigma)}}{\frac1n}
=&\frac1{2R^3}\frac{\frac1{B_n}\left(1+\frac1n-\frac2{n^2}\right)-\frac1{\sigma_s(R)-\tilde\sigma}}{\frac1n}
\\
=&\frac1{2R^3}\left[\frac1{B_n}+\frac{\frac1{B_n}-\frac1{\sigma_s(R)-\tilde\sigma}}{\frac1n}-
\frac2n\frac1{B_n}\right]
\\
=&\frac1{2R^3}\left[\frac1{B_n}+\lambda\frac{nQ_n(R)\left(\frac{Q_n'(R)}{Q_n(R)}-\frac nR\right)- \frac{\tilde\sigma}{\underline\sigma}n\frac{\rho^{n+2}}{R^{n+2}}Q_n'(\rho)}{B_n(\sigma_s(R)-\tilde\sigma)}
-\frac2n\frac1{B_n}\right].
\end{aligned}
\end{equation}
Notice that
$$
nQ_n(R)=R\left[Q_n'(R)-\left(Q_n'(R)-\frac nRQ_n(R)\right)\right].
$$
Thus, from \eqref{Qn-2} and \eqref{eq(4.23)}, we derive
\begin{equation}
\label{eq(4.25)}
\lim_{n\to\infty}nQ_n(R)=R.
\end{equation}
Hence, by Lemmas \ref{lem-1}, \ref{lem-2}, the assertion (i) and \eqref{eq(4.25)}, sending $n\to\infty$ in \eqref{eq(4.24)} yields \eqref{eq(4.5)}.
Therefore, there exists a positive integer $n^*>N_1$ such that for each $n\ge n^*$,
$$
\frac1{4R^3(\sigma_s(R)-\tilde\sigma)}\le
\frac{\frac{\mu_n}{n^3}-\frac{1}{2R^3(\sigma_s(R)-\tilde\sigma)}}
{\frac1n}\le\frac3{4R^3(\sigma_s(R)-\tilde\sigma)},
$$
i.e.,
$$
\frac{2n^3+n^2}{4R^3(\sigma_s(R)-\tilde\sigma)}\le\mu_n\le
\frac{2n^3+3n^2}{4R^3(\sigma_s(R)-\tilde\sigma)}.
$$
Subsequently, there holds
$$
\mu_{n+1}-\mu_n\ge\frac{4n^2+8n+3}{4R^3(\sigma_s(R)-\tilde\sigma)}>0
$$
for $n\ge n^*$, which completes the proof of the lemma.
\end{proof}

Denote
\begin{equation*}
\label{n**}
n^{**}=\min\left\{n: n\ge n^*, \mu_n>\max\left\{\frac{A_n}{B_n}:B_n\neq0,n=0,1,\cdots, n^*-1\right\}\right\}.
\end{equation*}
Recall that we have computed
$$
[F_{\tilde{R}}(0,\mu)]Y_{n,0}=(A_n-\mu B_n)Y_{n,0},
\quad n=0,1,2,\cdots.
$$
Thus, from the definition of $n^{**}$, we derive that for even $n\ge n^{**}$,
$$
{\rm Ker}[F_{\tilde{R}}(0,\mu_n)]={\rm span}\{Y_{n,0}\},
$$
i.e.,
$$
{\rm dim}({\rm Ker}[F_{\tilde{R}}(0,\mu_n)])=1.
$$
Next, since
$$
[F_{\tilde{R}}(0,\mu_n)]Y_{k,0}=(A_k-\mu_n B_k)Y_{k,0},\quad k=0,2,4,\cdots,
$$
and $A_k-\mu_n B_k\neq0$ for $k=0,2,\cdots,n-2,n+2,\cdots$,
$$
Y_1={\rm Im}[F_{\tilde{R}}(0,\mu_n)]={\rm span}\{Y_{0,0},Y_{2,0},\cdots,Y_{n-2,0},Y_{n+2,0},
\cdots\},
$$
namely,
$$
{\rm codim}Y_1=1.
$$
Finally, it is easy to see that
$$
[F_{\mu\tilde{R}}(0,\mu_n)]Y_{n,0}=-B_nY_{n,0}\not\in Y_1.
$$

To sum up, we have the following result by the Crandall-Rabinowitz theorem.

\begin{theorem}
\label{thm-2}
There exists a positive integer $n^{**}$ such that for every even integer $n\ge n^{**}$, $\mu_n$
defined by \eqref{eq(4.21)} is a bifurcation point of the symmetry-breaking solutions to the system \eqref{eq(1.1)}-\eqref{eq(1.6)} with free boundary
$$
\partial\eps{\Omega}: r=R+\ep Y_{n,0}(\theta)+O(\ep^2),
\quad
\partial\eps{D}: r=\rho+\ep\frac{\lambda}{\underline\sigma} Q'_n(\rho)Y_{n,0}(\theta)+O(\ep^2),
$$
where $\lambda$, $Q_n$ are respectively given by \eqref{lambda}, \eqref{Qn}.
\end{theorem}

\section*{Acknowledgments}
This work was done in part while the first author was visiting Department of Applied and Computational
Mathematics and Statistics, University of Notre Dame, under the support of China Scholarship Council
(201708360059). The kind hospitality and the financial support are gratefully acknowledged. This work is also supported by the National Natural Science Foundation of China (No. 11601200, No. 11861038
and No. 11771156), and the Science and Technology Planning Project from  Educational Commission of Jiangxi Province, China (No. GJJ160299).

\end{document}